\documentclass[a4paper,12pt]{amsart}
\usepackage{version}
\usepackage{amssymb}
\usepackage{amsfonts}
\usepackage{graphicx}
\usepackage{tikz}
\usetikzlibrary{arrows}

\usepackage{xcolor}
\usepackage[pagebackref]{hyperref}
\hypersetup{
   colorlinks,
    linkcolor={red!60!black},
    citecolor={blue!60!black},
    urlcolor={blue!90!black}
}

\usepackage{epstopdf}
\usepackage[english]{babel}
\usepackage{float}
\usepackage{cite}
\usepackage{tikz,pgf}
\usetikzlibrary{patterns,spy,decorations.pathreplacing,angles,quotes}
\usepackage{tikz-cd}

\usepackage{geometry}
\newtheorem{theorem}{Theorem}[section]

\newtheorem{lemma}[theorem]{Lemma}

\newtheorem{proposition}[theorem]{Proposition}

\newtheorem{claim}[theorem]{Claim}

\theoremstyle{definition}

\theoremstyle{remark}
\newtheorem{remark}[theorem]{Remark}

\numberwithin{equation}{section}

\renewcommand\bigskip{\medskip}

\def\to{\rightarrow}

\def\cF{\mathcal{F}}
\def\N{\mathbb N}

\def\Q{\mathbb Q}

\def\R{\mathbb R}

\def\dlp{d_{\rm LP}}

\def\Z{\mathbb Z}

\def\cF{\mathcal{F}}
\def\cE{\mathcal{E}}

\def\bmu{\bar{\mu}}

\def\cN{\mathcal{N}}

\def\-1{^{-1}}

\def\bmu{\bar{\mu}}
\def\bnu{\bar{\nu}}

\newcommand{\io}{\mathtt{i}}
\newcommand{\ko}{\mathtt{k}}
\newcommand{\jo}{\mathtt{j}}

\DeclareMathOperator{\dimh}{\dim_H}

\usepackage{tabularx}
\usepackage{booktabs}
\begin{document}

\title[Projections of diagonal self-affine measures]{The dimension of projections of planar diagonal self-affine measures}
\author{Aleksi Py\"or\"al\"a}
\address[
Aleksi Py\"or\"al\"a]
        {Department of Mathematics and Statistics \\ 
        P.O.\ Box 35 (MaD) \\ 
        FI-40014 University of Jyväskylä \\ 
         Finland}
         \email{aleksi.pyorala@gmail.com}
\thanks{The work on this project has been supported by the Research Council of Finland via the project \emph{GeoQuantAM:
Geometric and Quantitative Analysis on Metric spaces}, grant no. 354241. I thank Alex Rutar and Meng Wu for enlightening discussions on diagonal self-affine measures, and the referee for a careful reading of the manuscript and for many useful comments.}
\date{\today}
\begin{abstract}
     We show that if $\mu$ is a self-affine measure on the plane defined by an iterated function system of contractions with diagonal linear parts, then under an irrationality assumption on the entries of the linear parts, $$
     \dimh \mu \circ \pi^{-1}= \min\lbrace 1,\dimh\mu\rbrace
     $$
     for any non-principal orthogonal projection $\pi$. 
\end{abstract}
\maketitle

\section{Introduction}

The study of the size of orthogonal projections of sets and measures is an active and classical research topic in geometric measure theory and fractal geometry. Denoting by $\pi_\theta$ the orthogonal projection from $\R^2$ to the line spanned by $(1,2^\theta)$, the expected phenomenon is that for any Borel measure $\mu$ on $\R^2$, the set of $\theta\in\R$ for which 
\begin{equation}\label{eq-dimensiondrop}
\dimh \mu\circ\pi_\theta^{-1} < \min \lbrace 1, \dimh \mu \rbrace
\end{equation}
should be small, where $\dimh \mu = \inf \lbrace \dimh E:\ \mu(E)>0\rbrace$. While a classical result of Marstrand asserts that for any Borel measure $\mu$ this is indeed the case in the sense that the set of $\theta$'s for which \eqref{eq-dimensiondrop} can happen (the ``exceptional set'') has zero length, in this generality it is difficult to go any further; For example, the exceptional set may easily have large Hausdorff dimension. On the other hand, if we impose any additional structural assumptions on $\mu$, it is reasonable to expect that using this structure we should be able to say more about the exceptional set of $\theta$'s, perhaps regarding its dimension or, in the extreme case, even determine the set exactly. 

According to folklore, for many sets or measures arising from \emph{dynamical systems} the set of $\theta$'s satisfying \eqref{eq-dimensiondrop} should be explicitly determinable. A breakthrough in this line of research was achieved by Hochman and Shmerkin \cite{HochmanShmerkin2012} who proved that for a large class of dynamically defined measures on the plane, including planar self-similar measures with strong separation, \eqref{eq-dimensiondrop} cannot hold for \emph{any} $\theta\in\R$ if the associated iterated function system satisfies an \emph{irrationality condition}, that is, it involves a rotation which generates a dense subgroup of $SO_2(\R)$. The role of such a condition is to prevent exact aligntment of cylinders in any direction and is necessary for the exceptional set to be empty, in general. The analogous result has since then been verified for numerous other measures as well, including self-similar \cite{FalconerJin2014} and self-conformal measures \cite{BruceJin2019} with no separation assumptions at all, and many self-affine measures under various separation assumptions \cite{FalconerKempton2017, BruceJin2022, BaranyKaenmakiPyoralaWu2023, FergusonFraserSahlsten2015, BaranyHochmanRapaport2019}. 

Recall that an affine iterated function system $\lbrace \varphi_i(x) = A_i x + a_i\rbrace_{i\in\Gamma}$ on $\R^2$ is called diagonal if for each $i\in\Gamma$, 
$$
A_i = \begin{bmatrix}
    \lambda_1(i) & 0 \\ 0 & \lambda_2(i)
\end{bmatrix}
$$
for some $\lambda_1(i), \lambda_2(i) \in (-1,1)\setminus\lbrace 0\rbrace$, and that a measure $\mu$ is called self-affine if 
$$
\mu = \sum_{i\in\Gamma} p_i \mu\circ \varphi_i^{-1}
$$
for some $p_i>0$ with $\sum_{i\in\Gamma} p_i = 1$. In this paper we obtain a full resolution to the problem of Hausdorff dimension of projections of planar diagonal self-affine measures with an irrationality condition on the elements of $\lbrace A_i\rbrace_{i\in\Gamma}$. Previous partial results for such measures include the projection theorems of Ferguson et. al. \cite[Theorem 1.4]{FergusonFraserSahlsten2015} for self-affine measures on Bedford-McMullen carpets, and of Bárány et. al. \cite[Theorem 1.6]{BaranyKaenmakiPyoralaWu2023} for planar diagonal self-affine measures with the strong separation condition. In this paper we write $\pi_\theta\mu := \mu\circ\pi_\theta^{-1}$. 

\begin{theorem}\label{theorem-main}
    Let $\Phi=\lbrace \varphi_i(x) = A_i x + a_i\rbrace_{i\in\Gamma}$ be a diagonal affine iterated function system on $\R^2$. Suppose that 
    \begin{enumerate}
    \item there exists $l\in\Gamma$ such that $|\lambda_1(l)|\neq |\lambda_2(l)|$, and
    \item there exist $(s,t)\in \lbrace 1,2\rbrace^2$ and $(i,j)\in\Gamma^2$ such that $\frac{\log|\lambda_s(i)|}{\log|\lambda_t(j)|}\not\in\Q$.
    \end{enumerate}
    Then for any self-affine measure $\mu$ associated to $\Phi$ and for any $\theta\in\R$, 
    $$
    \dimh \pi_\theta\mu=\min\lbrace 1,\dimh\mu\rbrace.
    $$
\end{theorem}
Condition (1) of Theorem \ref{theorem-main} ensures that $\mu$ is not self-similar, and is necessary for the statement to hold. We remark that while Theorem \ref{theorem-main} relaxes the separation assumption from \cite[Theorem 1.6]{BaranyKaenmakiPyoralaWu2023}, there the authors also prove something stronger for the measures in question, namely that they have so-called \emph{uniform scaling sceneries}. Proving this property for self-affine measures in the absence of any separation conditions remains a challenging open problem. Finally, we remark that the conclusion of Theorem \ref{theorem-main} is not true for orthogonal projections to the coordinate axes. Counter-examples are found, for example, by choosing the contractions in $\Phi$ in such a way that their fixed points lie on a vertical or horizontal line. 

\subsection{On the proof}
For simplicity, suppose that $\mu$ has simple Lyapunov spectrum, and that $y$-axis is the major asymptotic contracting direction. As the majority of the previous works on the topic, also the present paper relies on the method of local entropy averages that Hochman and Shmerkin introduced in \cite{HochmanShmerkin2012}. This method makes it possible to bound $\dimh \pi_\theta\mu$ by bounding the \emph{finite-scale entropies} of the measures $\pi_\theta\mu_k$, where $(\mu_k)_{k\in\N}$ are \emph{magnifications} of $\mu$ along a suitable filtration. As in the work of Falconer and Jin \cite{FalconerJin2014} on projections of measures on self-similar sets with no separation conditions, the first step towards bypassing the requirement for separation conditions in our setting is taken by magnifying the measure $\mu$ (or rather, the corresponding Bernoulli measure $\bmu$) in the symbolic space $\Gamma^\N$ instead of $\R^2$, along a properly chosen filtration of $\Gamma^\N$. The elements of the filtration we choose produce approximate squares when projected onto the plane through the canonical projection $\Pi$, and relying on the non-conformal structure of $\mu$, we show that the canonical projections of the magnifications $\mu_k$ of $\bmu$ along this filtration have a product structure which is essentially of the form
\begin{equation}\label{product}
\Pi\mu_k = A_k(\pi_{\mathtt{x}}\mu)_{I_k} \times \nu_k
\end{equation}
where $A_k:\R^2\to\R^2$ is affine, $\pi_\mathtt{x}:\R^2\to\R$ denotes the orthogonal projection to the $x$-axis, $I_k\subset \R$ is a short interval and $\nu_k$ is a conditional measure of $\mu$ supported on a vertical line segment. Crucially, 
$$
\dimh((\pi_{\mathtt{x}}\mu)_{I_k}\times\nu_k) = \dimh \pi_\mathtt{x}\mu+\dimh \nu_k = \dimh \mu
$$
by a dimension conservation result of Feng and Hu \cite{FengHu2009}. Similar product structures were also observed in \cite{FergusonFraserSahlsten2015, BaranyKaenmakiPyoralaWu2023} for the measures in question. The representation \eqref{product} is made precise in Proposition \ref{prop-productstructure}.

Given this product structure of magnifications, our goal is now to bound the finite-scale entropies of the measures
\begin{equation}\label{productmeasure}
\pi_\theta A_k(\pi_{\mathtt{x}}\mu)_{I_k} \times \nu_k, \qquad \theta\in\R,\ k\in\N.
\end{equation}
Here $\pi_{\mathtt{x}}\mu$ is a self-similar measure on $\R$, whence the problem of bounding the size of \eqref{productmeasure} is related to the problems studied in \cite{BruceJin2022, BaranyKaenmakiPyoralaWu2023}: Were the conditional measure $\nu_k$ in \eqref{productmeasure} also self-similar (which it nearly is, in a certain dynamical sense), then we would know that the Hausdorff dimension of \eqref{productmeasure} would equal $\dimh \mu$ for every $\theta$. However, even knowing the dimension of \eqref{productmeasure} is of no immediate use, since we need a bound for the finite-scale entropies of \eqref{productmeasure} which is \emph{uniform over most} $k$. In \cite{BaranyKaenmakiPyoralaWu2023} and \cite{FergusonFraserSahlsten2015}, by relying on additional separation conditions the authors managed to show that the sequence
\begin{equation}\label{eq-sequence}
(A_k(\pi_{\mathtt{x}}\mu)_{I_k} \times \nu_k)_{k\in\N}
\end{equation}
equidistributes for a well-structured distribution supported on measures whose projections have large entropy, which then led to the desired uniform bound for the entropies of \eqref{productmeasure}. In our setting, \eqref{eq-sequence} does not seem to equidistribute for any obvious distribution since the length of $I_k$ decreases to $0$ and its location is difficult to control in $k$. This forces us to find a uniform lower bound for the entropies of $\pi_\theta A(\pi_\mathtt{x}\mu)_I \times \nu$ directly, where $\nu$ is a conditional measure of $\mu$ and the parameters $A:\R^2\to \R^2$ and $I\subseteq \R$ range over all affine maps and intervals such that ${\rm diam}\,A(I\times [-1,1])\approx 1$. This bound is achieved with the help of Proposition \ref{prop-productprojections}, the key technical tool of the paper, which yields a uniform lower bound for the small-scale entropies of $\pi_\theta(\pi_\mathtt{x}\mu\times\nu)$ as $\theta$ ranges over a large interval.

\begin{table}[H]
\caption{Notation}
\begin{tabularx}{\textwidth}{@{}p{0.4\textwidth}X@{}}
  \toprule
  $\Phi = \lbrace \varphi_i(x) = A_i x + a_i\rbrace_{i\in\Gamma}$ & System of affine invertible contractions \\
  $\lambda_1(i), \lambda_2(i)$ & The diagonal elements of $A_i$\\
  $\bar{\mu}$ & Bernoulli measure on $\Gamma^\N$ \\
  $\Pi$ & The natural projection $\Gamma^\N \to\R^2$ \\
  $\mu$   & Projection of $\bar{\mu}$ through $\Pi$  \\
  $\lambda_1^\mu, \lambda_2^\mu$ & Lyapunov exponents of $\bmu$\\
  $\mu_D$ & Normalized restriction on $D$ \\
  $\mu^D$ & Measure $\mu_D$ linearly rescaled onto $[-1,1)^d$ \\
  $\pi_\theta:\R^2\to \R$ & $\pi_\theta(x,y) = x+ 2^\theta y$ \\ 
  $\pi_\mathtt{x}, \pi_\mathtt{y}:\R^2\to \R$ & $\pi_\mathtt{x}(x,y) = x$, $\pi_\mathtt{y}(x,y)= y$\\
  $\mu_{\io}$ & A measure supported on $[-1,1]$ from the disintegration $\mu = \int \delta_{\pi_\mathtt{x}\Pi(\io)} \times \mu_\io \,d\bmu(\io)$ \\
  $S_t, T_x$ & Scaling by $2^t$ and translation by $-x$, respectively\\
  $\cE_n = \lbrace E_n(\io):\ \io\in\Gamma^\N\rbrace$ & Partition of $\Gamma^\N$ such that elements of $\Pi(\cE_n)$ are approximate squares of side length $2^{-n}$\\
    $t_n = t_n(\io)$ & $\min\lbrace k\in \N:\ 2^{-n} \geq \lambda_2(\io|_k)\rbrace$ \\
    $B^\Pi(\io,r)$ & $\Pi^{-1}(B(\Pi(\io),r)) \subseteq \Gamma^\N$\\
    $\Gamma_\io$ & $\Pi^{-1}(\pi_\mathtt{x}^{-1}(\pi_\mathtt{x}(\Pi(\io))))$, the ``symbolic vertical slice through $\io$''\\
    $\tilde{\Pi}$ & $\Gamma^\N\times\Gamma^\N\to \R^2$, $\tilde{\Pi}(\jo, \ko) = (\pi_\mathtt{x} \Pi(\jo), \pi_\mathtt{y} \Pi(\ko))$ \\
    $\tau_k = \tau_k(\jo, \ko,\theta)$ &  $\max\lbrace n\in\N: |\lambda_2(\ko|_n)\geq 2^{-\theta}|\lambda_1(\jo|_k)|\rbrace$\\
    $\lbrace F_n^\theta(\jo,\ko): \jo, \ko\in\Gamma^\N\rbrace$ & Partition of $\Gamma^\N\times\Gamma^\N$ such that $\Pi(F_n^\theta(\jo,\ko))$ is a rectangle of width $|\lambda_1(\jo|_n)|$ and height $\approx 2^{-\theta}|\lambda_1(\jo|_n)|$ \\
    $Z$ & $\lbrace (\io,t):\ \io \in \Gamma^\N, 0\leq t \leq -\log |\lambda_2(i_0)|\rbrace$ \\
    $G: Z \to \mathcal{P}(\R^2)$ & $G(\io,t) = \pi_\mathtt{x} \mu \times S_t\mu_\io$\\
  \bottomrule
\end{tabularx}
\end{table}
\newpage

\section{Preliminaries}
For $t\in \R$, let $S_t:\R^d\to\R^d$ denote the map $x\mapsto 2^t x$. For $\theta\in\R$, we let $\pi_\theta:\R^2\to\R$ denote the map $(x,y) \mapsto x+ 2^{\theta}y$. Up to an affine change of coordinates, $\pi_\theta$ is the orthogonal projection to the line spanned by $(1,2^\theta)$. For a measure $\mu$ and a measurable set $B$, let $\mu|_B(\cdot) := \mu(\cdot\cap B)$ denote the restriction of $\mu$ on $B$ and if $\mu(B)>0$, let $\mu_B := \mu(B)^{-1}\mu|_B$.  

\subsection{Entropy}

For each $n\in\N$, let
$$
\mathcal{D}_n = \mathcal{D}_n(\R) = \lbrace [k 2^{-n}, (k+1)2^{-n}):\ k\in \Z\rbrace 
$$
denote the partition of $\R$ into dyadic intervals of length $2^{-n}$. For a Borel measure $\mu$ on $\R$ and a partition $\cE$ of $\R$, let 
$$
H(\nu, \cE) := -\sum_{E\in\cE}\mu(E)\log\mu(E)
$$
denote the entropy of $\mu$ with respect to the partition $\cE$. For entropy with respect to the dyadic partition, we write $H_n(\mu) := H(\mu, \mathcal{D}_n)$. It is well-known (see \cite[Theorem 1.3]{FanLauRao2002}) that for any Borel probability measure $\mu$, 
$$
\dimh\mu \leq \liminf_{n\to\infty} \frac{1}{n}H_n(\mu).
$$
For two partitions $\cE$ and $\cF$, we denote by $H(\nu, \cE|\cF) := \sum_{F\in\cF} \nu(F) H(\nu_F, \cE)$ the conditional entropy of $\mu$ with respect to $\cE$ given $\cF$. Let also $\cE\vee\cF := \lbrace E\cap F:\ E\in\cE, F\in\cF\rbrace$ denote the join of $\cE$ and $\cF$. Below, we record some elementary properties of entropy, cf. \cite[Section 2]{CoverThomas2006}.
\begin{lemma}[Concavity]\label{lemma-concavityofentropy}
If $\mu_1,\ldots, \mu_m$ are Borel probability measures on $\R$, $\cE$ and $\cF$ are partitions of $\R$ and $p_1,\ldots, p_m$ are non-negative reals with $\sum_{i=1}^m p_i = 1$, then 
    $$
H\left( \sum_{i=1}^m p_i \mu_i, \cE|\cF\right) \geq \sum_{i=1}^m p_i H(\mu_i, \cE|\cF).
    $$
\end{lemma}
\begin{lemma}[Continuity]\label{lemma-continuityofentropy}
    If $\cE$ and $\cF$ are partitions of $\R$ such that each $E\in\cE$ intersects at most $k$ elements of $\mathcal{F}$ and vice versa, then 
    $$
|H(\mu,\cE) - H(\mu, \cF)| \leq \log k.
    $$
\end{lemma}
 
\begin{lemma}[Chain rule]\label{lemma-chainrule}
    If $\cE$ and $\cF$ are partitions of $\R$, then for any Borel probability measure $\mu$, 
    $$
H(\mu, \cE \vee \cF) = H(\mu, \cF) + H(\mu, \cE|\cF).
    $$
\end{lemma}

\subsection{Self-affine measures}

Let $\Gamma$ be a finite set with $\#\Gamma\geq 2$, and let $\Phi = \lbrace \varphi_i(x) = A_i x + a_i\rbrace_{i\in\Gamma}$ be a collection of affine invertible contractions of $\R^2$, where 
$$
A_i = \begin{bmatrix}
    \lambda_1(i) & 0\\
    0 & \lambda_2(i)
\end{bmatrix}
$$
with $\lambda_1(i), \lambda_2(i)\in\R$ for each $i\in\Gamma$. We let $\Gamma^* = \bigcup_{n=0}^\infty \Gamma^n$ denote the collection of finite words composed of characters of $\Gamma$, where $\Gamma^0 := \lbrace \emptyset\rbrace$. For $\io= (i_0, i_1,\ldots)\in\Gamma^\N$ and $n=1,2,\ldots$, let $\io|_n \in\Gamma^{n}$ denote the projection to the first $n$ coordinates. We equip $\Gamma^\N$ with the topology generated by the cylinder sets
$$
[\io] = \lbrace \jo\in\Gamma^\N:\ \jo|_{|\io|} = \io\rbrace,
$$
$\io\in\Gamma^*$, where $|\io|$ denotes the unique integer for which $\io\in\Gamma^{|\io|}$.
For $\jo= (j_0,\ldots, j_n)\in\Gamma^*$, we write $\varphi_\jo := \varphi_{j_0}\circ\cdots\circ \varphi_{j_n}$ and for $j\in\lbrace 1,2\rbrace$, $\lambda_j(\jo) = \lambda_j(j_0)\cdots \lambda_j(j_n)$. We let $\sigma:\Gamma^\N\to\Gamma^\N$ denote the left-shift, $\sigma(i_0,i_1,\ldots) = (i_1,i_2,\ldots)$.

Let $\Pi:\Gamma^\N\to\R^2$ denote the canonical projection,
$$
\Pi(\io) = \lim_{n\to\infty} \varphi_{\io|_n}(0).
$$
If $\bmu$ is a Bernoulli probability measure on $\Gamma^\N$, then the measure $\mu:=\Pi\bmu$ on $\R^2$ is called self-affine and can easily be seen to satisfy the equation
$$
\mu = \sum_{i\in\Gamma} \bmu([i]) \varphi_i \mu.
$$
It follows from Kingman's ergodic theorem that for $\bmu$-almost every $\io\in\Gamma^\N$, the limits
$$
0>\lambda_j := \lambda_j^\mu := \lim_{n\to\infty} \frac{\log |\lambda_j(\io|_n)|}{n}
$$
for $j=1,2$ exist. The numbers $\lambda_1$ and $\lambda_2$ are called the \emph{Lyapunov exponents} of $\mu$. We say that $\bmu$ (or $\mu$) has \emph{simple Lyapunov spectrum} if $\lambda_1\neq\lambda_2$. If $\lambda_i >\lambda_j$, then it is easy to see that
$$
\lim_{n\to\infty} |\lambda_j(\io|_n)||\lambda_i(\io|_n)|^{-1} = 0
$$
for $\bmu$-almost every $\io\in\Gamma^\N$. We recall the following result regarding the dimension conservation for diagonal self-affine measures, which follows from the Ledrappier-Young formula due to Feng and Hu \cite[Theorem 2.11]{FengHu2009}; see also\cite[Proof of Theorem 1.7 and Remark 6.3]{Feng2019-preprint} for more discussion on how \cite[Theorem 2.11]{FengHu2009} implies Theorem \ref{dimensionconservation}.
\begin{theorem}[Corollary of Theorem 2.11 of \cite{FengHu2009}]\label{dimensionconservation}
    Let $\mu$ be a diagonal self-affine measure on $\R^2$, let $\pi_\mathtt{x}:\R^2\to\R$ denote the orthogonal projection to the $x$-axis, and let $\mu = \int \mu_x\,d\pi_\mathtt{x}\mu(x)$ be the disintegration of $\mu$ with respect to $\pi_\mathtt{x}$. Then for $\pi_\mathtt{x}\mu$-almost every $x$, 
    $$
\dimh \mu = \dimh \pi_\mathtt{x}\mu + \dimh \mu_x.
    $$
\end{theorem}

\subsection{Suspension flows}

Let $f: \Gamma^\N \to [0,+\infty)$ be a locally Hölder continuous function. The set
$$
Z = \lbrace (\io, t):\ \io\in\Gamma^\N,\ 0\leq t\leq f(\io)\rbrace
$$
equipped with the identification $(\io,f(\io)) = (\sigma \io, 0)$ is called the \emph{suspension of} $\Gamma$ \emph{under} $f$. The suspension $Z$ together with the action $(\mathcal{T}_s)_{s\geq 0}$ induced on $Z$ by the non-negative reals, given by $\mathcal{T}_s(\io,t) = (\io,t+s)$, is called the \emph{suspension semi-flow} of $\Gamma^\N$ over $f$. If $\mu$ is a $\sigma$-invariant and ergodic measure on $\Gamma^\N$ (that is, $\sigma\mu = \mu$ and the only invariant functions are the almost everywhere constant functions), then it is easy to see that the measure $(\mu\times\mathcal{L})_Z$ is invariant and ergodic under the action $(\mathcal{T}_s)_{s\geq 0}$.
\begin{proposition}\label{prop-suspensionergodic}
    Let $(Z, (\mathcal{T}_s)_{s\geq0}, (\mu\times\mathcal{L})_Z)$ be an ergodic suspension semi-flow. Suppose that $(\io,0)$ is a periodic point of $(\mathcal{T}_s)_{s\geq0}$ with period $\alpha>0$. For any $\beta>0$ which is not a rational multiple of $\alpha^{-1}$, the discrete-time dynamical system $(Z, \mathcal{T}_\beta, (\mu\times\mathcal{L})_Z)$ is ergodic.
\end{proposition}
\begin{proof}
    It is proven for example in \cite[Lemma 3.11]{Hochman2012} that the system $(Z, \mathcal{T}_\beta, (\mu\times\mathcal{L})_Z)$ is ergodic if and only if no integer multiple of $\beta$ is an eigenvalue of $(Z, (\mathcal{T}_s)_{s\geq 0}, (\mu\times\mathcal{L})_Z)$. Recall that $\beta$ is an eigenvalue of $(Z, (\mathcal{T}_s)_{s\geq 0}, (\mu\times\mathcal{L})_Z)$ if there exists a measurable function $\psi$, called an eigenfunction, such that $\psi(\io,t) = e^{-2\pi i \beta t}\psi(\io,0)$ for $\bmu$-almost every $\io\in\Gamma^\N$ and every $t\geq 0$.
    
    For a contradiction, suppose that for some $n\in\Z$, $n\beta$ is an eigenvalue with eigenfunction $\psi$. By \cite[Proposition 6.2]{ParryPollicott1990} we may take $\psi$ to be continuous, in particular, $\psi$ is defined at the periodic point $(\io,0)$ and
    $$
    \psi(\io, 0) = \psi(\io, \alpha) = e^{-2\pi i n \beta \alpha} \psi(\io, 0).
    $$
    Because $|\psi|$ is invariant and therefore almost surely constant by ergodicity of the system $(Z,(\mathcal{T}_s)_{s\geq0},(\mu\times\mathcal{L})_Z)$, it follows that $2n\beta\alpha$ must be an integer. This is a contradiction.
\end{proof}

\section{Local entropy averages}

\subsection{The filtration}

In \cite{HochmanShmerkin2012}, Hochman and Shmerkin introduced a method which makes it possible to bound the Hausdorff dimension of a projection of a measure by bounding finite-scale entropies of magnifications of the measure along a suitably chosen filtration. We will next define the filtration along which we magnify. 

Let $\Phi = \lbrace \varphi_i\rbrace_{i\in\Gamma}$ be a system of affine contractions as in the statement of Theorem \ref{theorem-main}, and let $\mu$ be a self-affine measure associated to $\Phi$. Let $\bmu$ denote the associated Bernoulli measure on $\Gamma^\N$. Define the sets
\begin{align*}
\cE_n^1 &= \lbrace [\io]:\ \io\in\Gamma^*,\ |\lambda_2(\io)|>|\lambda_1(\io)|,\ 2^{-n}\geq |\lambda_1(\io)| \geq 2^{-n}\min_{i\in\Gamma}|\lambda_1(i)|\rbrace, \\
 \cE_n^2 &= \lbrace [\io]:\ \io\in\Gamma^*,\ |\lambda_1(\io)|\geq|\lambda_2(\io)|,\ 2^{-n}\geq |\lambda_2(\io)| \geq 2^{-n}\min_{i\in\Gamma}|\lambda_2(i)|\rbrace 
\end{align*}
for every $n\in\N$. Then clearly, for every $n$, $\cE_n^1 \cap \cE_n^2 = \emptyset$ and $\Gamma^\N = \bigcup_{[\io]\in \cE_n^1\cup \cE_n^2} [\io]$, where the union is disjoint. Note that for $[\io]\in\cE_n^j$, the set $\Pi([\io])\subset \R^2$ is a rectangle with shorter side of length $|\lambda_j(\io)| \approx 2^{-n}$. Here and in the following, we write $A\approx B$ to indicate that $A\leq O(B)$ and $B\leq O(A)$ where the implicit constants do not depend on $A$ or $B$. We proceed to refine the partitions $\cE_n^1\cup \cE_n^2$ further by joining them with either vertical or horizontal ``dyadic tubes'' of width $2^{-n}$, so that the canonical projections of elements of this join are ``approximate squares'', that is, rectangles of dimensions roughly $2^{-n}$. 

Define the maps $\pi_{\mathtt{x}}, \pi_{\mathtt{y}}:\R^2\to\R$ by $\pi_\mathtt{x}(x,y) = x$ and $\pi_\mathtt{y}(x,y) = y$ for $(x,y)\in\R^2$. Set $\cE_0 = \lbrace \emptyset\rbrace$ and for every $n\geq 1$, let
$$
\cE_n = (\cE_n^1 \vee \Pi^{-1}(\pi_\mathtt{y}^{-1}((\mathcal{D}_n(\R))))) \cup (\cE_n^2 \vee \Pi^{-1}(\pi_\mathtt{x}^{-1}((\mathcal{D}_n(\R))))) \vee \cE_{n-1}.
$$
It is easy to see that $(\cE_n)_{n\in\N}$ is an increasing sequence of partitions of $\Gamma^\N$. The reader should think of $\cE_n$ as a partition of $\Gamma^\N$ into approximate squares as in Figure \ref{fig:partition}: For every $\io\in\Gamma^\N$, if $E_n(\io)$ denotes the unique element of $\cE_n$ that contains $\io$, then for most $n$, the set $\Pi(E_n(\io))\subset \R^2$ is a rectangle of 
\begin{enumerate}
    \item[a)] width $|\lambda_1(\io|_n)|\approx 2^{-n}$ and height $\leq 2^{-n}$ in the case $|\lambda_2(\io|_n)| > |\lambda_1(\io|_n)|$, or
    \item[b)] width $\leq 2^{-n}$ and height $|\lambda_2(\io|_n)|\approx 2^{-n}$ in the case $|\lambda_1(\io|_n)|>|\lambda_2(\io|_n)|$.
\end{enumerate}
\begin{figure}
\begin{tikzpicture}
    \draw (0.3,0) -- (5.3,0) -- (5.3, 1) -- (0.3, 1) -- (0.3,0);
    \draw[dotted, thick] (1,0) -- (1,1);
    \draw[rectangle, fill=lightgray] (2,0) --(2,1) --(3,1) -- (3,0) -- (2,0);
    \draw[dotted, thick] (4,0) -- (4,1);
    \draw[dotted, thick] (5,0) -- (5,1);
    \draw[decoration={brace,raise=5pt},decorate] (0,0) --node[left = 6pt]{$|\lambda_2(\io|_n)|$} (0,1);
    \draw[decoration={brace,mirror,raise=5pt},decorate] (2,0) --node[below = 6pt]{$2^{-n}$} (3,0);
\end{tikzpicture}
\qquad
\begin{tikzpicture}
    \draw (0,0.3) -- (0, 3.3) -- (1,3.3) -- (1, 0.3) -- (0,0.3);
    \draw[rectangle, fill=lightgray] (0,1) -- (1,1) -- (1,2) -- (0,2) -- (0,1);
    \draw[dotted, thick] (0,3) -- (1,3);
    \draw[decoration={brace,mirror,raise=5pt},decorate] (0,0.3) --node[below = 6pt]{$|\lambda_1(\io|_n)|$} (1,0.3);
    \draw[decoration={brace,mirror,raise=5pt},decorate] (1,1) --node[right = 6pt]{$2^{-n}$} (1,2);
\end{tikzpicture}
\caption{The rectangle $\Pi(E_n(\io))$ (colored gray) in the case a) on the right and in the case b) on the left.}
\label{fig:partition}
\end{figure}
\begin{remark}
In the case when $\mu$ has simple Lyapunov spectrum, for $\bmu$-almost every $\io$ and every large enough $n$, the set $\Pi(E_n(\io))$ is always a rectangle of type a) (if $\lambda_1^\mu<\lambda_2^\mu$) or of type b) (if $\lambda_1^\mu>\lambda_2^\mu$). However, if $\lambda_1^\mu=\lambda_2^\mu$, then for infinitely many $n$ the side-lengths of the rectangle $\Pi([\io|_n])$ are comparable, whence we will witness a third case c) where it is difficult to say anything more than that $\Pi(E_n(\io))$ a rectangle of dimensions $\leq 2^{-n}$.
However, even then the case c) will only occur with a negligible frequency of $n$, whence it will do no harm to think of the cases a) and b) as the only possibilities. We include a brief expansion of this discussion in Subsection \ref{case-equal}.
\end{remark}

The following version of the local entropy averages of \cite{HochmanShmerkin2012} follows by combining \cite[Theorem 4.4, Proposition 5.2, Proposition 5.3 and Theorem 5.4]{HochmanShmerkin2012}. Recall that $\pi_\theta(x,y) = x+ 2^\theta y$. 

\begin{theorem}\label{thm-localentropyaverages}
    Let $\varepsilon>0$, $c\geq 0$ and $\theta\in\R$. There exists an integer $N_0\in\N$ such that if $N\geq N_0$ and
    $$
    \liminf_{n\to\infty}\frac{1}{n}\sum_{k=1}^n \frac{1}{N}H_N(\pi_\theta S_{kN} \Pi\bmu_{E_{kN}(\io)}) \geq c
    $$
    for $\bmu$-almost every $\io$, then $\dimh \pi_\theta\mu \geq c - \varepsilon$. 
\end{theorem}

\section{Proof of Theorem \ref{theorem-main}}

Without loss of generality, we suppose that the self-affine measure $\mu$ is supported on $[-1,1]^2$. We divide the proof into two parts according to the multiplicity of the Lyapunov exponents of $\mu$. We begin with the case $\lambda_1^\mu \neq \lambda_2^\mu$ since it is technically simpler, and the case $\lambda_1^\mu = \lambda_2^\mu$ in fact essentially reduces to the case $\lambda_1^\mu \neq \lambda_2^\mu$ (as we see in Subsection \ref{case-equal}), using \cite[Theorem 1]{Robbins1953} which states that the number $|\lambda_1(\io|_n)||\lambda_2(\io|_n)|^{-1}$ is usually very far from one even when the Lyapunov exponents agree. Since the cases $\lambda_1^\mu>\lambda_2^\mu$ and $\lambda_2^\mu>\lambda_1^\mu$ are symmetric, we only consider the cases $\lambda_1^\mu > \lambda_2^\mu$ and $\lambda_1^\mu = \lambda_2^\mu$.

\subsection{The case $\lambda_1^\mu > \lambda_2^\mu$} 
For every $\io\in\Gamma^\N$, write 
$$
t_n = t_n(\io) = \min \lbrace k \in \N:\ 2^{-n}\geq|\lambda_2(\io|_k)|\rbrace.
$$
Recall that we write $E_n(\io)$ for the unique element of $\cE_n$ that contains $\io$. After possibly applying a random translation to $\mu$, the measure will give no mass to boundaries of the partition $\Pi^{-1}(\pi_\mathtt{y}^{-1}((\mathcal{D}_m(\R))))$ for any $m\in\N$. After such a translation, it readily follows from the assumption $\lambda_1^\mu > \lambda_2^\mu$ that for $\bmu$-almost every $\io$ and all large enough $n$, $E_n(\io) \in \cE_n^2 \vee \Pi^{-1}(\pi_\mathtt{x}^{-1}((\mathcal{D}_n(\R))))$ and so $\Pi(E_n(\io))$ is a rectangle of height $|\lambda_2(\io|_{t_n})|$ and width at most $2^{-n}$. See Figure \ref{fig:case1}.

\begin{figure}
    \begin{tikzpicture}
    \draw (0.3,0) -- (5.3,0) -- (5.3, 1) -- (0.3, 1) -- (0.3,0);
    \draw[dotted, thick] (1,0) -- (1,1);
    \draw[rectangle, fill=lightgray] (2,0) --(2,1) --(3,1) -- (3,0) -- (2,0);
    \draw[dotted, thick] (4,0) -- (4,1);
    \draw[dotted, thick] (5,0) -- (5,1);
    \draw[decoration={brace,raise=5pt},decorate] (0,0) --node[left = 6pt]{$|\lambda_2(\io|_{t_n})|$} (0,1);
    \draw[decoration={brace,mirror,raise=5pt},decorate] (2,0) --node[below = 6pt]{$2^{-n}$} (3,0);
\end{tikzpicture}
    \caption{Under the assumption $\lambda_1^\mu>\lambda_2^\mu$, the cylinders $\Pi([\io|_n])$ are ``vertically flat''. The ``approximate square'' $\Pi(E_n(\io))$ is colored gray.}
    \label{fig:case1}
\end{figure}

We shall next collect some technical facts regarding the sequences $(E_n(\io))_{n\in\N}$ and $(t_n(\io))_{n\in\N}$.

\begin{lemma}\label{lemma-equidistribution}
    There exists a Borel probability measure $\bnu$ on $\Gamma^\N$ which is equivalent to $\bmu$ and the following holds: For any $N\in\N$ and any Borel set $A\subseteq \Gamma^\N$, for $\bmu$-almost every $\io\in\Gamma^\N$, 
    $$
\lim_{m\to\infty} \frac{1}{m}\sum_{n=1}^m \delta_{\sigma^{t_{nN}}\io} (A) = \bnu(A).
    $$
\end{lemma}
\begin{proof}
     Consider the suspension 
    $$
W = \lbrace (\io,t):\ \io\in\Gamma^\N, 0\leq t\leq -\log |\lambda_1(\io|_1)|\rbrace
    $$
    equipped with the identification $(\io, -\log |\lambda_1(\io|_1)|) = (\sigma\io, 0)$. By possibly changing the base of the logarithm, we may suppose that for at least one $i\in\Gamma$, the number $\log|\lambda_1(i)|$ is irrational. Let $i\in\Gamma$ be such, and note that the point $((i,i,i,\ldots), 0)$ is periodic with period $-\log|\lambda_1(i)|$. In particular, if $T_N: W \to W$ denotes the map $(\io,t) \mapsto (\io,t+N)$, then the dynamical system $(W, T_N, (\bmu\times\mathcal{L})_W)$ is ergodic by Proposition \ref{prop-suspensionergodic}. Since $(\io, t+nN) = (\sigma^{t_{nN}}\io, t+nN + \log|\lambda_1(\io|_{t_{nN}})|)$ for every $t\geq 0$ by the identification $(\io, -\log |\lambda_1(\io|_1)|) = (\sigma\io, 0)$ and the definition of $t_{nN}$, and the projection of $(\bmu\times\mathcal{L})_W$ to the first marginal is equivalent to $\bmu$, the statement follows from an application of Birkhoff's ergodic theorem for $(W, T_N, (\bmu\times\mathcal{L})_W)$.
\end{proof}

We let $B(x,r)$ denote the closed ball in the maximum metric of radius $r>0$ centered at $x$, so that when $x\in\R^2$ the set $B(x,r)$ is a closed square centered at $x$. We also write $B^\Pi(\io,r) := \Pi^{-1} B(\Pi(\io),r) \subseteq\Gamma^\N$ for any $\io\in\Gamma^\N$ and $r\geq 0$. 

\begin{lemma}\label{lemma-bigdensity1}
    For any $\varepsilon>0$ and large enough $N\in\N$, there exists $c>0$ such that the following holds: For $\bmu$-almost every $\io\in\Gamma^\N$, there exists a set $\cN_\varepsilon\subseteq \N$ such that $\liminf_{m\to\infty} m^{-1}\#(\cN_\varepsilon\cap[0,m]) \geq1-\varepsilon$ and for every $n\in\cN_\varepsilon$,
    $$
    \frac{\bmu(E_{nN}(\io))}{\bmu(B^{\Pi}(\io, 2^{-(n-1)N}) \cap [\io|_{t_{(n-1)N}}])} \geq c.
    $$
\end{lemma}
\begin{proof}
    Since Lebesgue-almost every number is normal in every base, it follows from Fubini's theorem that after randomly translating $\mu$, for any $N\in\N$ we have
    $$
    \lim_{m\to\infty} \frac{1}{m} \sum_{n=1}^m \delta_{2^{nN} \pi_\mathtt{x}x \mod 1} = \mathcal{L}_{[0,1]}
    $$
    for $\mu$-almost every $x$. In particular, for any large enough $N\in\N$ and $\bmu$-almost every $\io\in\Gamma^\N$, there exists a set $\cN_\varepsilon^1\subseteq \N$ such that $\liminf_{m\to\infty} m^{-1}\#(\cN_\varepsilon^1\cap[0,m]) \geq1-\varepsilon/2$ and for every $n\in\cN_\varepsilon^1$, $2^{nN} \pi_\mathtt{x}\Pi(\io) \mod 1 \in [2^{-N},1-2^{-N}]$ or equivalently,
    $$
 B(\pi_\mathtt{x}\Pi(\io), 2^{-(n+1)N})\subseteq \mathcal{D}_{nN}(\pi_\mathtt{x}\Pi(\io)).
    $$
    In particular, taking preimages under $\pi_\mathtt{x}$ and $\Pi$,
    \begin{equation}\label{eq-lemma4.2.1}
B^{\Pi}(\io, 2^{-(n+1)N})\cap [\io|_{t_{(n+1)N}}] \subseteq E_{nN}(\io)
    \end{equation}
    for $n\in\cN_\varepsilon^1$. 
    
    Let $\ell :=\left\lceil \max_{i,j\in\Gamma}\frac{\log|\lambda_1(i)|}{ \log|\lambda_2(j)|}\right\rceil$ so that $[\io|_{t_{m}\ell}] \subseteq B^\Pi(\io, 2^{-m}) \cap [\io|_{t_m}]$ for any $\io\in\Gamma^\N$ and $m\in\N$. Since $\limsup_{m\to\infty} \frac{-\log \bmu([\io|_m])}{m} \leq \log \#\Gamma$ for $\bmu$-almost every $\io\in\Gamma^\N$ and $t_m \leq O(m)$ for every $m$, we have
    \begin{align*}
&\limsup_{m\to\infty} \frac{1}{m}\sum_{n=1}^{m-1} \log \frac{\bmu(B^{\Pi}(\io, 2^{-(n-1)N})\cap [\io|_{t_{(n-1)N}}])}{\bmu(B^{\Pi}(\io, 2^{-(n+1)N})\cap [\io|_{t_{(n+1)N}}])} \\
\leq\ &2\limsup_{m\to\infty} \frac{-\log \bmu(B^{\Pi}(\io, 2^{-mN})\cap [\io|_{t_{mN}}])}{m} \\
\leq\ &\limsup_{m\to\infty} \frac{-\log \bmu([\io|_{\ell t_{mN}}])}{\ell t_{mN}} \cdot \frac{\ell t_{mN}}{m}\\
\leq\ &O(N)
    \end{align*}
   for $\bmu$-almost every $\io\in\Gamma^\N$, where the implicit constant in $O$ only depends on the IFS $\Phi$. In particular, there exists $c>0$ such that for $\bmu$-almost every $\io\in\Gamma^\N$, there exists a set $\cN_\varepsilon^2\subseteq \N$ such that $\liminf_{m\to\infty} m^{-1}\#(\cN_\varepsilon^2\cap[0,m]) \geq1-\varepsilon/2$ and for every $n\in\cN_\varepsilon^2$, 
   \begin{equation}\label{eq-lemma4.2.2}
   \frac{\bmu(B^{\Pi}(\io, 2^{-(n+1)N})\cap [\io|_{t_{(n+1)N}}])}{\bmu(B^{\Pi}(\io, 2^{-(n-1)N})\cap [\io|_{t_{(n-1)N}}])} \geq c.
   \end{equation}
   Setting $\cN_\varepsilon = \cN_\varepsilon^1\cap\cN_\varepsilon^2$ and combining \eqref{eq-lemma4.2.1} and \eqref{eq-lemma4.2.2} completes the proof. 
\end{proof}

For $\delta>0$, let 
$$
E_n^\delta(\io) = \lbrace \jo\in \Gamma^*:\ [\jo]\subseteq E_n(\io)\ \text{and}\ |\lambda_1(\jo)|\leq \delta 2^{-n} < |\lambda_1(\jo|_{|\jo|-1})|\rbrace
$$
index the family of cylinders that are completely contained in $E_n(\io)$ with ``relative width'' $\approx\delta$. Note that $E_n^\delta(\io) = \emptyset$ if $\Pi(E_n(\io))$ has width $<2^{-n}\delta$ but this happens rarely. The essential content of the following lemma is that the measure $\bmu$ is not usually concentrated near the boundary of $E_n(\io)$. In order to simplify notation, we occasionally identify $E_n^\delta(\io)$ with $\bigcup_{\jo\in E_n^\delta(\io)} [\jo]$ and use notation such as $\bmu(E_{nN}(\io)\setminus E_{nN}^\delta(\io))$.

\begin{lemma}\label{lemma-bigdensity2}
For every $\varepsilon>0$ and large enough $N\in\N$, there exists $\delta>0$ such that the following holds: For $\bmu$-almost every $\io\in\Gamma^\N$, there exists a set $\cN_\varepsilon\subseteq \N$ such that $\liminf_{m\to\infty} m^{-1}\#(\cN_\varepsilon\cap[0,m]) \geq1-\varepsilon$ and for every $n\in\cN_\varepsilon$,
$$
\frac{\bmu(E_{nN}^\delta(\io))}{\bmu(E_{nN}(\io))} \geq 1-\varepsilon.
$$
\end{lemma}
\begin{proof}
    It suffices to show that for any $\varepsilon'>0$ and $N\in\N$, there exists $\delta>0$ such that 
    \begin{equation}\label{eq-sufficestoshow}
\frac{\bmu(E_{nN}(\io)\setminus E_{nN}^\delta(\io))}{\bmu(B^{\Pi}(\io, 2^{-(n-1)N})\cap [\io|_{t_{(n-1)N}}])} \leq \varepsilon'
    \end{equation}
    for every $\io\in\Gamma^\N$ and $n\in\N$. Indeed, it is easy to see that
    $$
\frac{\bmu(E_{nN}^\delta(\io))}{\bmu(E_{nN}(\io))} \geq 1 - \frac{\bmu(E_{nN}(\io)\setminus E_{nN}^\delta(\io))}{\bmu(B^{\Pi}(\io, 2^{-(n-1)N})\cap [\io|_{t_{(n-1)N}}])} \frac{\bmu(B^{\Pi}(\io, 2^{-(n-1)N})\cap [\io|_{t_{(n-1)N}}])}{\bmu(E_{nN}(\io))}
    $$
    whence the claim will follow from \eqref{eq-sufficestoshow} and Lemma \ref{lemma-bigdensity1} by choosing $\varepsilon' < \varepsilon/c$.
    
    If the measure $\pi_\mathtt{x}\mu$ were atomic, the claim of the lemma would be trivial. Supposing that it is not, there exists $\delta>0$ such that if $I,J\subseteq \R$ are intervals of length at most $2\delta$, then $\mu(\pi_\mathtt{x}^{-1}(I\cup J))\leq \varepsilon'$. In particular, for small enough $\delta>0$ and any $\jo\in\Gamma^*$ with $|\lambda_1(\jo)|\geq 2^{-nN}$,
    $$
    \bmu_{[\jo]}(E_{nN}(\io)\setminus E_{nN}^\delta(\io)) = \mu( \varphi_{\jo}^{-1}(\Pi(E_{nN}(\io)\setminus E_{nN}^\delta(\io)))) < \varepsilon'
    $$
    since $\pi_\mathtt{x}(\varphi_{\jo}^{-1}(\Pi(E_{nN}(\io)\setminus E_{nN}^\delta(\io))))$ is contained in a union of two intervals of length at most $2\delta$, by the definition of $E_{nN}^\delta(\io)$.  
    
    Now, for every $\jo\in\Gamma^*$ such that $2^{-nN} \leq |\lambda_1(\jo)| \leq 2^{-(n-1)N}$ and $[\jo]\cap E_{nN}(\io)\neq \emptyset$, we have $[\jo]\subseteq B^{\Pi}(\io, 2^{-(n-1)N})\cap [\io|_{t_{(n-1)N}}]$, whence
    \begin{align*}
&\frac{\bmu(E_{nN}(\io)\setminus E_{nN}^\delta(\io))}{\bmu(B^{\Pi}(\io, 2^{-(n-1)N})\cap [\io|_{t_{(n-1)N}}])}  \\
\leq\ &\sum_{\jo} \frac{\bmu([\jo])}{\bmu(B^{\Pi}(\io, 2^{-(n-1)N})\cap [\io|_{t_{(n-1)N}}])}\bmu_{[\jo]}(E_{nN}(\io)\setminus E_{nN}^\delta(\io)) \\
\leq\ &\varepsilon' \sum_{\jo} \frac{\bmu([\jo])}{\bmu(B^{\Pi}(\io, 2^{-(n-1)N})\cap [\io|_{t_{(n-1)N}}])}\\
\leq\ &\varepsilon'
    \end{align*}
    for every $\io\in\Gamma^\N$ and $n\in\N$, where the sum is taken over those $\jo\in\Gamma^*$ with $2^{-nN} \leq |\lambda_1(\jo)| \leq 2^{-(n-1)N}$ and $[\jo]\cap E_{nN}(\io)\neq \emptyset$. This completes the proof.
\end{proof}

\subsubsection{Product structure} Recall that $\pi_{\mathtt{x}}:\R^2\to\R$ denotes the orthogonal projection to the $x$-axis. Let $\mu = \int \delta_{\pi_{\mathtt{x}}(\Pi(\io))} \times \mu_\io\,d\bmu(\io)$ be the disintegration of $\mu$ with respect to $\pi_{\mathtt{x}}$, where each $\mu_\io$ is a probability measure on $[-1,1]$. 

The content of the following proposition is that the measures $\Pi\bmu_{E_n(\io)}$ enjoy a product structure. We use $\dlp$ to denote the L{\'e}vy-Prokhorov distance between probability measures. 

\begin{proposition}[Product structure]\label{prop-productstructure}
    For any $\varepsilon>0$ and large enough $N\in\N$, the following holds for $\bmu$-almost every $\io\in\Gamma^\N$: For any $\varepsilon'>0$, there exists a set $\cN_\varepsilon\subseteq \N$ such that $\liminf_{m\to\infty} m^{-1}\#(\cN_\varepsilon\cap[0,m]) \geq1-\varepsilon$ and for every $n\in\cN_\varepsilon$,
    $$
\dlp( S_{nN}\Pi\bmu_{E_{nN}(\io)},\ S_{nN}(\varphi_{\io|_{t_{nN}}}(\pi_{\mathtt{x}}\mu \times \mu_{\sigma^{t_{nN}}\io}))_{\Pi(E_{nN}(\io))}) <\varepsilon'.
    $$
\end{proposition}

\begin{proof}
    The proof is similar to that of \cite[Lemma 6.3]{BaranyKaenmakiPyoralaWu2023}, with some technical differences arising from the fact that we are ``zooming in'' along the sets $E_n(\io)$ instead of Euclidean balls centered at $\Pi(\io)$. In this proof, if $R = [a,b]\times[c,d]\subset \R^2$ is a rectangle and $\nu$ is a Borel measure on $\R^2$, we let $\nu^R$ denote the push-forward of $\nu_R$ under the map $(x,y)\mapsto (x/(b-a), y/(d-c))$ which takes $R$ onto a translation of $[0,1]^2$. 

    It follows from self-affinity of $\mu$ and the Bernoulli property of $\bmu$ that for any $\jo\in\Gamma^*$,
    $$
    \Pi \bmu_{[\jo]} = \varphi_\jo \mu.
    $$
    Combining this with the equality $E_n(\io) = \Pi^{-1}(\Pi(E_n(\io)))\cap [\io|_{t_n}]$ we find that for any $\io\in\Gamma^\N$ and $n\in\N$,
    \begin{equation}\label{eq-selfaffinity}
\Pi\bmu_{E_n(\io)} = \Pi(\bmu_{[\io|_{t_n}]})_{E_n(\io)} = (\varphi_{\io|_{t_n}}\mu)_{\Pi(E_n(\io))} = \varphi_{\io|_{t_n}} \mu_{\varphi_{\io|_{t_n}}^{-1} \Pi(E_n(\io))}.
    \end{equation}
    Here the set $\varphi_{\io|_{t_n}}^{-1} \Pi(E_n(\io))$ is a rectangle centered at $\varphi_{\io|_{t_{n}}}^{-1} \Pi(\io) = \Pi(\sigma^{t_{n}}\io)$, of height $2|\lambda_2(\io|_{t_{n}})| |\lambda_2(\io|_{t_{n}})|^{-1} = 2$ and of width at most $2^{-n} |\lambda_1(\io_{t_{n}})|^{-1}$, which eventually allows us to relate $\mu_{\varphi_{\io|_{t_n}}^{-1} \Pi(E_n(\io))}$ to the conditional measure $\mu_{\sigma^{t_{n}}\io}$. In order to do this rigorously, it is convenient to first establish the equality
    \begin{equation}\label{eq-convenient}
        \frac{\pi_\mathtt{x}\mu(\pi_\mathtt{x}(\varphi_{\io|_{t_{n}}}^{-1} \Pi(E_{n}(\io))))}{\pi_\mathtt{x}\mu(B(\pi_\mathtt{x}\Pi(\sigma^{t_{n}}\io), 2^{-{n}} |\lambda_1(\io|_{t_{n}})|^{-1}))} = \frac{\bmu(E_{n}(\io))}{\bmu(B^{\Pi}(\io, 2^{-n})\cap [\io|_{t_{n}}])}
    \end{equation}
    for every $\io\in\Gamma^\N$ and $n\in\N$. To establish \eqref{eq-convenient}, we first note that by reasoning similarly as in \eqref{eq-selfaffinity} we have
    \begin{equation}\label{eq-sameasure} \bmu(E_{n}(\io)) 
= \bmu([\io|_{t_{n}}])\mu(\varphi_{\io|_{t_n}}^{-1}\Pi(E_{n}(\io))).
    \end{equation}
    Since $\varphi_{\io|_{t_n}}^{-1}\Pi(E_{n}(\io)) = \pi_\mathtt{x}^{-1}\pi_\mathtt{x} \varphi_{\io|_{t_n}}^{-1}\Pi(E_{n}(\io))$, this shows that the nominators of \eqref{eq-convenient} are equal up to the constant $\bmu([\io|_{t_{n}}])$. Regarding the denominators, first note that $\sigma(\Pi^{-1}(A) \cap [i]) = \Pi^{-1}( \varphi_i^{-1}(A))$ for any $A\subseteq \R^2$ and $i\in\Gamma$, whence
    \begin{align}\label{eq-sameasure2}
    \bmu(B^\Pi(\io,2^{-n})\cap [\io|_{t_{n}}])) &= \bmu([\io|_{t_{n}}]) \bmu(\sigma^{n}(\Pi^{-1} B(\Pi(\io),2^{-n}) \cap [\io|_{t_{n}}])) \nonumber\\
    &= \bmu([\io|_{t_{n}}]) \bmu(\Pi^{-1} \varphi_{\io|_{t_{n}}}^{-1} B(\Pi(\io, 2^{-n})))\nonumber \\
    &= \bmu([\io|_{t_{n}}]) \pi_\mathtt{x}\mu(B(\pi_\mathtt{x}\Pi(\sigma^{t_n}\io), 2^{-n} |\lambda_1(\io_{t_{n}})|^{-1})).
    \end{align}
    In the last equality we used the fact that 
    $$
    \varphi_{\io|_{t_{n}}}^{-1} B(\Pi(\io), 2^{-n}))) \cap [-1,1]^2 = B(\pi_\mathtt{x}\Pi(\sigma^{t_n}\io), 2^{-n} |\lambda_1(\io_{t_{n}})|^{-1})\times [-1,1]
    $$
    which is readily verified using the fact that $\Pi(E_{n}(\io)) \subseteq B(\Pi(\io), 2^{-n})$ and recalling how $\varphi_{\io|_{t_n}}^{-1}$ acts on the square $\Pi(E_{n}(\io))$. Combining \eqref{eq-sameasure} and \eqref{eq-sameasure2} yields \eqref{eq-convenient}.
    
    Now, Lemma \ref{lemma-bigdensity1} together with \eqref{eq-convenient} asserts that for any $\varepsilon>0$ and large enough $N\in\N$, there exists $c>0$ such that $\bmu$-almost every $\io\in\Gamma^\N$, there exists a set $\cN_\varepsilon^1\subseteq \N$ with $\liminf_{m\to\infty} m^{-1}\#(\cN_\varepsilon^1\cap[0,m]) \geq1-\varepsilon/2$ such that for every $n\in\cN_\varepsilon^1$, 
    \begin{equation}\label{eq-subsetoflargemeasure}
    \frac{\pi_\mathtt{x}\mu(\pi_\mathtt{x}(\varphi_{\io|_{t_{nN}}}^{-1} \Pi(E_{nN}(\io))))}{\pi_\mathtt{x}\mu(B(\pi_\mathtt{x}\Pi(\sigma^{t_{nN}}\io), 2^{-nN} |\lambda_1(\io|_{t_{nN}})|^{-1}))} \geq c.
    \end{equation}

    Let us fix such $\varepsilon, c$ and $N$ for the rest of the proof. Let 
    $$
    f_c(\io,r) := \sup_{U\subseteq B(\pi_\mathtt{x}\Pi(\io), r)} \dlp (\mu^{U\times [-1,1]},\ (\pi_\mathtt{x}\mu\times\mu_{\io})^{U\times [-1,1]})
    $$
    where the supremum is taken over those measurable sets $U\subseteq B(\pi_\mathtt{x}\Pi(\io), r)$ with $\pi_\mathtt{x}\mu(U) \geq c\cdot\pi_\mathtt{x}\mu(B(\pi_\mathtt{x}\Pi(\io), r))$. By the Lebesgue-Besicovitch differentiation theorem applied for the function $\io\mapsto \delta_{\pi_\mathtt{x}\Pi(\io)}\times \mu_\io$, for $\bmu$-almost every $\io\in\Gamma^\N$ we have
    \begin{equation}\label{eq-closetoaslice}
\lim_{r\to 0}  f_c(\io,r) =0.
   \end{equation}
   Under the assumption $\lambda_1^\mu>\lambda_2^\mu$, $\lim_{n\to\infty}2^{-n}|\lambda_1(\io|_{t_n})|^{-1} = 0$ as $n\to\infty$ for $\bmu$-almost every $\io\in\Gamma^\N$. Therefore it follows from \eqref{eq-closetoaslice} and Lemma \ref{lemma-equidistribution} that 
   $$
\lim_{m\to\infty}\frac{1}{m}\sum_{n=1}^m f_c(\sigma^{t_{nN}}\io, 2^{-nN} |\lambda_1(\io|_{t_{nN}})|^{-1}) = 0
   $$
for $\bmu$-almost every $\io\in\Gamma^\N$. In particular, for any $\varepsilon'>0$ there exists a set $\cN_\varepsilon^2\subseteq \N$ with $\liminf_{m\to\infty} m^{-1}\#(\cN_\varepsilon^2\cap[0,m]) \geq1-\varepsilon/2$ such that for every $n\in\cN_\varepsilon^2$, 
\begin{equation}\label{eq-closetoaslice2}
f_c(\sigma^{t_{nN}}\io, 2^{-nN} |\lambda_1(\io|_{t_{nN}})|^{-1}) < \varepsilon'.
\end{equation}
    Therefore, combining \eqref{eq-subsetoflargemeasure} and \eqref{eq-closetoaslice2}, for $\bmu$-almost every $\io\in\Gamma^\N$ and every $n\in\cN_\varepsilon := \cN_\varepsilon^1\cap \cN_\varepsilon^2$ we have
    \begin{align}\label{eq-closetoaslice3}
    &\dlp (\mu^{\varphi_{\io|_{t_{nN}}}^{-1} \Pi(E_{nN}(\io)))},\ (\pi_{\mathtt{x}}\mu\times \mu_{\sigma^{t_{nN}}\io})^{\varphi_{\io|_{t_{nN}}}^{-1} \Pi(E_{nN}(\io)))}) \\
    \leq\ &f_c(\sigma^{t_{nN}}\io, 2^{-nN} |\lambda_1(\io|_{t_{nN}})|^{-1})\nonumber\\
    \leq\ &\varepsilon'.\nonumber
    \end{align}
    Since the operations 
    $$
    (\cdot)^{\varphi_{\io|_{t_{nN}}}^{-1} \Pi(E_{nN}(\io))}\qquad \text{and}\qquad S_{nN}  \varphi_{\io|_{t_{nN}}}(\cdot)_{\varphi_{\io|_{t_{nN}}}^{-1} \Pi(E_{nN}(\io))}
    $$
    differ only by a translation and a horizontal scaling, it follows from \eqref{eq-closetoaslice3} and \eqref{eq-selfaffinity} that for every $n\in\cN_\varepsilon$,
    $$
\dlp (S_{nN} \Pi\bmu_{E_{nN}(\io)},\ S_{nN} \varphi_{\io|_{t_{nN}}}(\pi_{\mathtt{x}}\mu  \times \mu_{\sigma^{t_{nN}}\io})_{\varphi_{\io|_{t_{nN}}}^{-1} \Pi(E_{nN}(\io))}) \leq \varepsilon'.
    $$
    This completes the proof.
\end{proof}

As the final preparation for the proof of Theorem \ref{theorem-main}, we present a lower bound for the entropies of projections of $\pi_{\mathtt{x}}\mu\times\mu_\io$. 

\begin{proposition}\label{prop-productprojections}
    Let $\mu$ be as in the statement of Theorem \ref{theorem-main}. For $\bmu$-almost every $\io\in\Gamma^\N$ and any $\varepsilon,M>0$, there exists $N_0 \in\N$ such that for any $N\geq N_0$ and any $\theta\in[0,M]$, 
    $$
\frac{1}{N} H_N(\pi_\theta(\pi_{\mathtt{x}}\mu\times \mu_\io)) \geq \min \lbrace 1, \dimh\mu\rbrace - \varepsilon.
    $$
\end{proposition}

The weight of the proposition lies in the fact that the number $N$ works simultaneously for all $\theta$ in a large interval. Since its proof is somewhat independent of that of the main result and requires some work, we first conclude the proof of Theorem \ref{theorem-main} and then proceed to prove Proposition \ref{prop-productprojections}.

\begin{proof}[Proof of Theorem \ref{theorem-main} assuming $\lambda_1^\mu>\lambda_2^\mu$]
Let $\mu$ be a self-affine measure associated to an IFS $\Phi = \lbrace \varphi_i\rbrace_{i\in\Gamma}$ as in the statement. We first point out that it suffices to prove the theorem for $\theta = 0$. Indeed, for any $\theta\in\R$, $\pi_\theta(x,y) = \pi_0(x, S_\theta y)$, and the push-forward of $\mu$ through $(x,y)\mapsto (x, S_{\theta}y)$ is again a self-affine measure which satisfies the conditions of the theorem. In this proof, it is convenient to consider the quantity 
$$
\tilde{H}_N(\mu) := \int_0^1 H_N(\delta_x*\mu)\,dx
$$
in place of $H_N(\mu)$ since the function $\mu\mapsto \tilde{H}_N(\mu)$ is continuous. By Lemma \ref{lemma-continuityofentropy}, $|\tilde{H}_N(\mu) - H_N(\mu)|\leq O(1)$ for any probability measure $\mu$. 

Let $\varepsilon>0$, and let $m, N\in\N$ be large with respect to $\varepsilon$; through the course of the proof, we will see how large they have to be. Recall that by Theorem \ref{thm-localentropyaverages} it is enough to find a lower bound for the entropy of the measure $\pi_0 S_{nN} \Pi \bmu_{E_n(\io)}$, for almost every $\io$ and ``most'' $n$. Applying Proposition \ref{prop-productstructure} with small enough $\varepsilon'$ and using continuity of $\tilde{H}_N(\mu)$, we have
\begin{align}\label{eq-firstestimate}
&\frac{1}{m}\sum_{n=1}^m \frac{1}{N} \tilde{H}_N(\pi_0 S_{nN} \Pi\bmu_{E_{nN}(\io)}) \nonumber \\
\geq\ &\frac{1}{m}\sum_{n=1}^m \frac{1}{N} \tilde{H}_N\left(\pi_0S_n(\pi_{\mathtt{x}}\varphi_{\io|_{t_n}}\mu \times \varphi_{\io|_{t_n}}\mu_{\sigma^{t_n}\io})_{\Pi(E_n(\io))})\right) - 2\varepsilon
\end{align}
for $\bmu$-almost every $\io\in\Gamma^\N$, when $N$ and $m$ are large enough. Just by iterating the relation $\mu = \int \varphi_{\io|_1} \mu\,d\bmu(\io)$, we have 
\begin{align*}
&(\pi_{\mathtt{x}}\varphi_{\io|_{t_n}}\mu \times \varphi_{\io|_{t_n}}\mu_{\sigma^{t_n}\io})_{\Pi(E_n(\io))} \\
&\qquad = \bmu(E_n(\io))^{-1} \int_{\jo\in E_{nN}(\io)} (\pi_{\mathtt{x}}\varphi_{\jo|_{n(\jo)}}\mu \times \varphi_{\io|_{t_n}}\mu_{\sigma^{t_n}\io}))|_{\Pi(E_{nN}(\io))} \,d\bmu(\jo)
\end{align*}
for any measurable $\jo \mapsto n(\jo)\in\N$ such that $[\jo|_{n(\jo)}] \subseteq [\io|_{t_n}]$. By Lemma \ref{lemma-bigdensity2}, those $\jo$ in the area of integration which do not belong to $\bigcup_{\jo \in E_{nN}^\delta(\io)}[\jo]$ carry negligible measure for most $n$, as long as $\delta>0$ is small enough. In particular, 
\begin{align*}
\frac{1}{m}\sum_{n=1}^m \dlp\big(& S_n(\pi_{\mathtt{x}}\varphi_{\io|_{t_n}}\mu \times \varphi_{\io|_{t_n}}\mu_{\sigma^{t_n}\io})_{\Pi(E_n(\io))},\\
&\qquad\qquad\sum_{\jo\in E_{nN}^\delta(\io)} q_\jo \cdot S_{nN} (\pi_{\mathtt{x}}\varphi_{\jo}\mu \times \varphi_{\io|_{t_n}}\mu_{\sigma^{t_n}\io})\big) < \varepsilon
\end{align*}
where $q_\jo := \frac{\bmu([\jo])}{\bmu(E_{nN}^\delta(\io))}$, for $\bmu$-almost every $\io\in\Gamma^\N$ and large enough $m$. Plugging this into \eqref{eq-firstestimate} and using concavity of entropy, we find that
\begin{align}\label{eq-concavityapplication}
&\frac{1}{m}\sum_{n=1}^m \frac{1}{N} \tilde{H}_N(\pi_0 S_{nN} \Pi\bmu_{E_{nN}(\io)}) \nonumber \\
\geq\ &\frac{1}{m}\sum_{n=1}^m \sum_{\jo\in E_{nN}^\delta(\io)}  \frac{q_\jo}{N} \tilde{H}_N\left(\pi_0 S_{nN} (\pi_{\mathtt{x}}\varphi_{\jo}\mu \times \varphi_{\io|_{t_{nN}}}\mu_{\sigma^{t_{nN}}\io}))\right)-3\varepsilon.
\end{align}
Here
\begin{align}\label{eq-notranslations}
    &\frac{1}{N} \tilde{H}_N\left(\pi_0 S_{nN}(\pi_{\mathtt{x}}\varphi_{\jo}\mu \times \varphi_{\io|_{t_{nN}}}\mu_{\sigma^{t_{nN}}\io})_{\Pi(E_{nN}(\io))})\right) \nonumber \\
    &\qquad \geq \frac{1}{N}\tilde{H}_N\left(\pi_0 S_{nN} (S_{\log|\lambda_1(\jo)|}\pi_{\mathtt{x}}\mu \times S_{\log|\lambda_2(\io|_{t_{nN}})|}\mu_{\sigma^{t_{nN}}\io})\right) - O(1/N)
\end{align}
for each $\jo\in E_{nN}^\delta(\io)$, using Lemma \ref{lemma-continuityofentropy} to get rid of the translations involved in $\varphi_\jo$ and $\varphi_{\io|_{t_n}}$. Next we apply the identity
\begin{align}\label{eq-piidentity}
&\pi_0 S_{nN} (S_{\log|\lambda_1(\jo)|} x , S_{\log|\lambda_2(\io|_{t_{nN}})|}y) \nonumber \\
&\qquad = S_{nN + \log|\lambda_1(\jo)|}\circ\pi_{\log|\lambda_2(\io|_{t_{nN}})|-\log|\lambda_1(\jo)|}(x,y)
\end{align}
which is immediate from the definition of $\pi_\theta$. Inserting both \eqref{eq-notranslations} and \eqref{eq-piidentity} into \eqref{eq-concavityapplication}, we find that
\begin{align}\label{eq-randomdirection}
&\frac{1}{m}\sum_{n=1}^m \frac{1}{N} \tilde{H}_N(\pi_0 S_{nN} \Pi\bmu_{E_{nN}(\io)}) \\
\geq\ &\frac{1}{m}\sum_{n=1}^m  \sum_{\jo\in E_{nN}^\delta(\io)}\frac{q_\jo}{N} \tilde{H}_N(S_{nN + \log|\lambda_1(\jo)|}(\pi_{\log|\lambda_2(\io|_{t_{nN}})|-\log|\lambda_1(\jo)|} (\pi_{\mathtt{x}}\mu \times \mu_{\sigma^{t_{nN}}\io}))) - 4\varepsilon.\nonumber
\end{align}
Since $\log \delta \leq nN + \log|\lambda_1(\jo)| \leq 0$ for every $\jo\in E_{nN}^\delta(\io)$, by Lemma \ref{lemma-continuityofentropy} we may replace the scaling map $S_{nN + \log|\lambda_1(\jo)|}$ by the identity in \eqref{eq-randomdirection}, at the cost of adding one more $\varepsilon$ to the right-hand side. Writing 
$$
\theta(\io,\jo,n):= \log|\lambda_2(\io|_{t_{nN}})|-\log|\lambda_1(\jo)|,
$$ 
we have ended up with the estimate
\begin{align}\label{eq-conclusion}
    &\frac{1}{m}\sum_{n=1}^m \frac{1}{N} \tilde{H}_N(\pi_0 S_{nN} \Pi\bmu_{E_{nN}(\io)}) \nonumber \\
\geq\ &\frac{1}{m}\sum_{n=1}^m  \sum_{\jo\in E_{nN}^\delta(\io)}\frac{q_\jo}{N} \tilde{H}_N(\pi_{\theta(\io,\jo,n)} (\pi_{\mathtt{x}}\mu \times \mu_{\sigma^{t_{nN}}\io})) - 5\varepsilon
\end{align}
for $\bmu$-almost every $\io\in\Gamma^\N$ and every large enough $m$. Observe that 
\begin{alignat*}{2}
\theta(\io,\jo,n) &\geq -n + \min_{i\in\Gamma}\log |\lambda_2(i)|^{-1} + n -\log \delta &&> 0\ \text{and} \\
\theta(\io,\jo,n) &\leq -n + n + \max_{i\in\Gamma}\log |\lambda_2(i)|^{-1} -\log \delta &&\leq -2\log\delta
\end{alignat*}
as long as $\delta$ is small enough. 

By Proposition \ref{prop-productprojections} and Egorov's theorem, if $N$ is large enough, there exists a set $F\subseteq \Gamma^\N$ such that $\bmu(F)\geq 1-\varepsilon$ and for every $\io\in F$ and $\eta\in [0,-2\log\delta]$, 
\begin{equation}\label{eq-projectionset}
\frac{1}{N} \tilde{H}_N(\pi_\eta(\pi_{\mathtt{x}}\mu\times\mu_\io))\geq \min \lbrace 1,\dimh\mu \rbrace-\varepsilon.
\end{equation}
Continuing from \eqref{eq-conclusion},
\begin{align}
&\frac{1}{m}\sum_{n=1}^m  \sum_{\jo\in E_{nN}^\delta(\io)}\frac{q_\jo}{N} \tilde{H}_N(\pi_{\theta(\io,\jo,n)} (\pi_{\mathtt{x}}\mu \times \mu_{\sigma^{t_{nN}}\io})) \nonumber\\
\geq\ &\frac{1}{m}\sum_{n=1}^m \delta_{\sigma^{t_{nN}}\io}(F)  \sum_{\jo\in E_{nN}^\delta(\io)}\frac{q_\jo}{N} \tilde{H}_N(\pi_{\theta(\io,\jo,n)} (\pi_{\mathtt{x}}\mu \times \mu_{\sigma^{t_{nN}}\io})) \nonumber\\
\geq\ &\frac{1}{m}\sum_{n=1}^m \delta_{\sigma^{t_{nN}}\io}(F)  \sum_{\jo\in E_{nN}^\delta(\io)} q_\jo (\min\lbrace 1,\dimh \mu\rbrace - \varepsilon)  \label{eq-keypart2}\\
\geq\ &\min\lbrace 1,\dimh\mu\rbrace - 3\varepsilon.\label{eq-keypart3} 
\end{align} 
In \eqref{eq-keypart2} we used \eqref{eq-projectionset} and in \eqref{eq-keypart3} we used Lemma \ref{lemma-equidistribution} and the fact that $\sum_{\jo\in E_{nN}^\delta(\io)} q_\jo = 1$. By \eqref{eq-conclusion} and Theorem \ref{thm-localentropyaverages}, this completes the proof of Theorem \ref{theorem-main}.
\end{proof}

\begin{remark} 
It is the estimate \eqref{eq-keypart2} in which it is crucial that Proposition \ref{prop-productprojections} gives the desired lower bound \emph{uniformly} for every $\theta$ in an arbitrarily large fixed interval. Indeed, the exact value of $\theta(\io,\jo,n)$ is something we do not know how to control as $\jo$ varies in $E_{nN}^\delta(\io)$. 
\end{remark}

\subsection{The case $\lambda_1^\mu= \lambda_2^\mu$}\label{case-equal}

In the case $\lambda_1^\mu > \lambda_2^\mu$, the cylinders $\Pi([\io|_n])$ were ``vertically flat'' for almost all $\io$ and all large enough $n$, which allowed us to deduce the product structure of the measures $\Pi\bmu_{E_n(\io)}$ in Proposition \ref{prop-productstructure}. In the case $\lambda_1^\mu = \lambda_2^\mu$ this ``uniform vertical flatness'' is no longer the case; However, \cite[Theorem 1]{Robbins1953} asserts that the sequence
$$
\left( \log \frac{|\lambda_1(\io|_k)|}{|\lambda_2(\io|_k)|} \right)_k= \left(\sum_{\ell=1}^k \log \frac{|\lambda_1(\sigma^\ell \io|_1)|}{|\lambda_2(\sigma^\ell \io|_1)|}\right)_k
$$
equidistributes (in a sense) for the Lebesgue measure on $\R$, in particular, for any $M>0$, 
$$
\lim_{n\to\infty} \frac{1}{n}\#\left\lbrace 1\leq k \leq n:\ \left| \log \frac{|\lambda_1(\io|_k)|}{|\lambda_2(\io|_k)|}\right| \leq M \right\rbrace = 0.
$$
In other words, the number $|\lambda_1(\io|_k)||\lambda_2(\io|_k)|^{-1}$ (the ``eccentricity'' of the rectangle $\Pi([\io|_k])$) is usually far from one in the sense that for any $M>0$, with full frequency, either 
\begin{equation}\label{eq-large}
|\lambda_1(\io|_n)| > 2^{M}|\lambda_2(\io|_n)|\ \text{or}\ |\lambda_2(\io|_n)| > 2^{M}|\lambda_1(\io|_n)|.
\end{equation}
In particular, with full frequency, the rectangle $\Pi([\io|_n])$ is either vertically or horizontally flat, and so we are essentially reduced to a combination of the cases $\lambda_1^\mu > \lambda_2^\mu$ and $\lambda_2^\mu > \lambda_1^\mu$. Indeed, in Proposition \ref{prop-productstructure}, for every $n$ we have to switch between $\pi_\mathtt{x}$ and $\pi_\mathtt{y}$ depending on which one of the numbers $|\lambda_1(\io|_n)|$ and $|\lambda_2(\io|_n)|$ dominates in \eqref{eq-large}. Proposition \ref{prop-productprojections} does not assume anything regarding the Lyapunov exponents of $\mu$ and so works also if $\pi_\mathtt{x}$ is replaced by $\pi_\mathtt{y}$ and $\mu_\io$ by a conditional measure supported on a horizontal line segment. Regarding the proof of Theorem \ref{theorem-main}, one first duplicates it for the case $\lambda_2^\mu> \lambda_1^\mu$, and then in the case $\lambda_1^\mu=\lambda_2^\mu$, the average $\frac{1}{m}\sum_{\ell=1}^m \frac{1}{N} H_N(\pi_\theta \Pi\bmu_{E_{nN}(\io)})$ is split into two parts: The average of those terms in which $\lambda_1(\io|_{nN}) \gg \lambda_2(\io|_{nN})$ is handled as in the case $\lambda_1^\mu > \lambda_2^\mu$, and the average of those terms in which $\lambda_2(\io|_{nN}) \gg \lambda_1(\io|_{nN})$ is handled as in the case $\lambda_2^\mu> \lambda_1^\mu$. We leave the details to the interested reader.

\section{Proof of Proposition \ref{prop-productprojections}}

It remains to prove Proposition \ref{prop-productprojections}. We begin by introducing some additional notation. Let $\bmu = \int \bmu_\io \,d \bmu(\io)$ denote the disintegration of $\bmu$ with respect to $\pi_{\mathtt{x}}\circ \Pi: \Gamma^\N \to \R$. Note that by uniqueness of disintegration, for almost every $\io$, $\Pi \bmu_\io = \delta_{\pi_\mathtt{x}\Pi(\io)}\times \mu_\io$. Let $\Gamma_\io := \Pi^{-1}(\pi_\mathtt{x}^{-1}(\pi_\mathtt{x}((\Pi(\io)))))\subset \Gamma^\N$ be the ``symbolic vertical slice through $\io$'', so that ${\rm supp}\,\bmu_\io\subseteq \Gamma_\io$. Let $\tilde{\Pi}:\Gamma^\N\times \Gamma^\N \to \R^2$ denote the map $\tilde{\Pi}(\jo,\ko) = (\pi_{\mathtt{x}} \Pi(\jo), \pi_{\mathtt{y}} \Pi(\ko))$, where $\pi_{\mathtt{y}}:\R^2\to\R$ denotes the orthogonal projection to the $y$-axis (identified with $\R$). With this notation, the following equality is immediate for $\bmu$-almost every $\io\in\Gamma^\N$:
$$
\tilde{\Pi} (\bmu \times \bmu_\io) = \pi_{\mathtt{x}}\mu\times\mu_\io.
$$

We will next define partitions of $\Gamma^\N\times\Gamma^\N$ similar to the partitions $\cE_n$ of the previous section. However, it is now more convenient to build for every $\theta\in\R$ a partition which produces rectangles of eccentricity $\approx 2^{-\theta}$ instead of approximate squares: Define the stopping time $\tau_\ell$ on $\Gamma^\N\times \Gamma^\N\times \R$,  
\begin{equation}\label{eq-stoppingtime}
    \tau_\ell = \tau_\ell(\jo,\ko,\theta) = \max\lbrace n\in\N:\ |\lambda_2(\ko|_n)|\geq 2^{-\theta}|\lambda_1(\jo|_\ell)| \rbrace
\end{equation}
for every $\ell\in\N$. For every $(\jo, \ko)\in\Gamma^\N\times \Gamma^\N$ and $n\in \N$, write $F_n^\theta(\jo,\ko) = [\jo|_{n}]\times [\ko|_{\tau_n}]$. The set $\tilde{\Pi} (F_n^\theta(\jo,\ko)) \subset \R^2$ is a rectangle of width $|\lambda_1(\jo|_n)|$ and height $|\lambda_2(\ko|_{\tau_n})|\approx 2^{-\theta}|\lambda_1(\jo|_n)|$.

Most of the work in proving Proposition \ref{prop-productprojections} goes towards proving the following local variant. 
\begin{proposition}\label{prop-localvariant}
For any $\varepsilon,M>0$, large enough $n\in\N$ and small enough $\delta>0$, there exist $m(\delta)\in\N$ and a set $U\subseteq \Gamma^\N\times\Gamma^\N$ with $\bmu\times\bmu(U)\geq 1-\delta$ such that for any $(\jo,\ko)\in U$, $\theta\in [0,M]$ and $m\geq m_0$,
    $$
\frac{1}{m} \sum_{\ell=1}^m \frac{1}{n} H_n(\pi_\theta S_{-\log|\lambda_1(\jo|_{\ell})|} \tilde{\Pi} (\bmu \times \bmu_{\ko})_{F_\ell^\theta(\jo,\ko)}) \geq \min\lbrace 1,\dimh\mu\rbrace-\varepsilon.
$$
\end{proposition}

\begin{proof}[Proof of Proposition \ref{prop-localvariant}]
For every $\theta$, we have the following description of the magnifications of the measure $\bmu\times\bmu_\io$. 

\begin{claim}\label{productrepresentation}
For $\bmu$-almost every $\io\in\Gamma^\N$, for $\bmu\times \bmu_\io$-almost every $(\jo,\ko)\in\Gamma^\N\times \Gamma_\io$ and for every $\ell\in \N$, there exists a translation $T_\ell$ such that 
    $$
S_{-\log |\lambda_1(\jo|_\ell)|} T_\ell \tilde{\Pi} (\bmu \times \bmu_\io)_{F_\ell^\theta(\jo,\ko)} =  \pi_{\mathtt{x}}\mu \times S_{\log |\lambda_2(\ko|_{\tau_\ell})|-\log |\lambda_1(\jo|_\ell)|}\mu_{\sigma^{\tau_\ell}\ko}.
$$
Here $\log |\lambda_2(\ko|_{\tau_\ell})| - \log |\lambda_1(\jo|_\ell)| \in [-\theta, -\theta+ \max_{i\in\Gamma} \log |\lambda_2(i)|^{-1}].$
\end{claim}
See Figure \ref{fig:productmeasure}.
\begin{figure}[H]
    \begin{tikzpicture}
    \draw (0.3,0) -- (6.3,0) -- (6.3, 2) -- (0.3, 2) -- (0.3,0);
    \draw (3.3, 1) node{$ T_\ell S_{-\log |\lambda_1(\jo|_\ell)|}\tilde{\Pi} (\bmu \times \bmu_\io)_{F_\ell^\theta(\jo,\ko)}$};
    \draw[decoration={brace,raise=5pt},decorate] (0,0) --node[left = 6pt]{$\approx S_{-\theta}\mu_{\sigma^{\tau_\ell}\ko}$} (0,2);
    \draw[decoration={brace,mirror,raise=5pt},decorate] (0.3,0) --node[below = 6pt]{$\pi_\mathtt{x}\mu$} (6.3,0);
\end{tikzpicture}
    \caption{The statement of Claim \ref{productrepresentation}. The measure $S_{-\log |\lambda_1(\jo|_\ell)|}T_\ell\tilde{\Pi} (\bmu \times \bmu_\io)_{F_\ell^\theta(\jo,\ko)}$ is a product of $\pi_\mathtt{x}\mu$ and $S_{\log |\lambda_2(\ko|_{\tau_\ell})|-\log |\lambda_1(\jo|_\ell)|}\mu_{\sigma^{\tau_\ell}\ko} \approx S_{-\theta}\mu_{\sigma^{\tau_\ell}\ko}$.}
    \label{fig:productmeasure}
\end{figure}
\begin{proof}[Proof of Claim]
The measures $\bmu_\io$ enjoy a form of dynamical self-similarity, namely
\begin{equation}\label{eq-dynamicalselfsimilarity}
(\bmu_\io)_{[i]} = (\bmu_{\sigma\io} \circ \sigma)_{[i]}
\end{equation}
for any $i\in\Gamma$ and $\bmu$-almost every $\io \in\Pi^{-1}(\pi_{\mathtt{x}}^{-1}(\pi_{\mathtt{x}}(\Pi([i]))))$. This follows from the equality
\begin{align*}
\bmu &= \sum_{i\in\Gamma} (\bmu\circ\sigma)|_{[i]} = \sum_{i\in\Gamma} \int (\bmu_\io\circ \sigma)|_{[i]}\,d\bmu(\io) \\
& \hspace{2cm}= \sum_{i\in\Gamma} \int (\bmu_\io\circ \sigma)|_{[i]}\,d\sigma\bmu(\io) = \sum_{i\in\Gamma} \int (\bmu_{\sigma\io}\circ \sigma)|_{[i]}\,d\bmu(\io)
\end{align*}
together with the uniqueness of disintegration. Taking also into account that we have $\sigma((\Pi^{-1}(A) \cap [i]) = \Pi^{-1}(\varphi_i^{-1}(A))$ for any $A\subseteq \R^2$, \eqref{eq-dynamicalselfsimilarity} asserts that
\begin{equation}\label{claimeq-1}
\Pi (\bmu_{\io})_{[i]} = (\bmu_{\sigma\io} \circ \sigma)_{[i]}\circ\Pi^{-1} = \varphi_i \Pi\bmu_{\sigma\io} = \varphi_i  (\delta_{\pi_\mathtt{x}\Pi(\sigma\io)} \times \mu_{\sigma \io}).
\end{equation}
On the other hand, by self-similarity of $\pi_\mathtt{x}\mu$ we have 
\begin{equation}\label{claimeq-2}
\Pi\bmu_{[i]} = \varphi_i \mu.
\end{equation}
Since $\tilde{\Pi} (\bmu \times \bmu_\io)_{F_\ell^\theta(\jo,\ko)} = \pi_\mathtt{x}\Pi \bmu_{[\jo|_n]}\times\pi_{\mathtt{y}}\Pi (\bmu_\io)_{[\ko|_{\tau_\ell}]}$, the first statement follows by combining \eqref{claimeq-1} and \eqref{claimeq-2}. The second is immediate from the definition of $\tau_\ell$.
\end{proof}

Recall that $\pi_\theta(x,y) = x + 2^\theta y$. By Claim \ref{productrepresentation}, we have 
\begin{align}\label{eq-projection}
&\pi_\theta S_{-\log |\lambda_1(\jo|_\ell)|} T_\ell\tilde{\Pi} (\bmu \times \bmu_\io)_{F_\ell^\theta(\jo,\ko)} \nonumber\\
=\ &\pi_0 ( \pi_{\mathtt{x}}\mu \times S_{\log |\lambda_2(\ko|_{\tau_\ell})|-\log |\lambda_1(\jo|_\ell)|+\theta}\mu_{\sigma^{\tau_\ell}\ko}).
\end{align}
Next we want to show that under the irrationality assumption of Theorem \ref{theorem-main}, the sequence 
$$
(\sigma^{\tau_\ell}\ko, \log|\lambda_2(\ko|_{\tau_\ell})|- \log |\lambda_1(\jo|_{\ell})|+\theta)_{\ell\in\N}.
$$
equidistributes for a measure absolutely continuous with respect to $\bmu\times\mathcal{L}$. This allows us to eventually bound the dimension of \eqref{eq-projection} using Marstrand's projection theorem. To establish this equidistribution, we define the set
$$
Z = \lbrace (\io, t):\ \io\in\Gamma^\N, 0\leq t \leq -\log |\lambda_2(\io|_1)|\rbrace
$$
equipped with the identification $(\io, -\log|\lambda_2(\io|_1)|) = (\sigma\io, 0)$. The reason for choosing such an identification is that it gives the equality 
\begin{equation}\label{eq-identification}
(\ko, \theta- \log |\lambda_1(\jo|_\ell)|) = (\sigma^{\tau_\ell}\ko,  \log |\lambda_2(\ko|_{\tau_\ell})| - \log |\lambda_1(\jo|_\ell)|+\theta)
\end{equation}
for every $\jo,\ko$ and $\ell$, by the definition of $\tau_\ell$. Let $(\mathcal{T}_s)_{s\geq 0}$ denote the action of non-negative reals on $Z$ given by $\mathcal{T}_s(\io,t) = (\io, t+s)$, and recall that $(\bmu\times \mathcal{L})_Z$ is invariant under $(\mathcal{T}_s)_{s\geq 0}$. 

\begin{claim}\label{claim-equidistribution}
Under the irrationality assumption of Theorem \ref{theorem-main}, the measure $\bmu\times(\bmu\times\mathcal{L})_Z$ on $\Gamma^\N\times Z$ is ergodic under the map $\sigma^*: (\jo, (\ko, t))\mapsto (\sigma \jo, (\ko, t- \log |\lambda_1(\jo|_1)|))$. 
\end{claim}
\begin{proof}[Proof of Claim]
By the assumption, there exists a pair $(i,j)\in\Gamma$ such that $\frac{\log|\lambda_2(i)|}{\log|\lambda_1(j)|}\not\in \Q$. The point $((i,i,i,\ldots), 0)\in Z$ is periodic with period $-\log|\lambda_2(i)|$, so by Proposition \ref{prop-suspensionergodic}, the measure $(\bmu\times\mathcal{L})_{Z}$ is ergodic under the map $\mathcal{T}_{-\log|\lambda_1(j)|}$. It now follows from \cite[Theorem 3]{Kakutani1951} that the skew-product system $(\Gamma^\N \times Z, \sigma^*, \bmu\times (\bmu\times\mathcal{L})_Z)$ is ergodic. 
\end{proof}

Define the function $G: Z \to \mathcal{P}(\R^2)$, 
$$
G(\io,t) = \pi_0(\pi_{\mathtt{x}}\mu \times S_t\mu_\io)
$$
and note that by \eqref{eq-projection}, \eqref{eq-identification} and Claim \ref{productrepresentation},
\begin{align}\label{eq-Gproperty}
G(\ko, \theta - \log |\lambda_1(\jo|_\ell)|) &= \pi_0(\pi_{\mathtt{x}}\mu \times S_{\log |\lambda_2(\ko|_{\tau_\ell})|- \log |\lambda_1(\jo|_\ell)|+\theta}\mu_{\sigma^{\tau_\ell}\ko}) \nonumber\\
&= \pi_\theta S_{-\log|\lambda_1(\jo|_{\ell})|} T_\ell \tilde{\Pi} (\bmu \times \bmu_{\ko})_{F_\ell^\theta(\jo,\ko)}.
\end{align}
By Marstrand's projection theorem for measures, Theorem \ref{dimensionconservation} and Fubini, we have for $(\bmu\times\mathcal{L})_Z$-almost every $(\io,t)$ that 
$$
\dimh\pi_0(\pi_{\mathtt{x}}\mu\times S_t \mu_\io) = \min\lbrace 1,\dimh\mu\rbrace.
$$
Let $\varepsilon>0$. Combining the well-known fact that $\dimh\nu \leq \liminf_{n\to\infty}\frac{1}{n}H_n(\nu)$ for any Borel probability measure $\nu$ with Egorov's theorem, we find a set $F\subseteq Z$ with $(\bmu\times \mathcal{L})_{Z}(F)>1-\varepsilon$ and an integer $N = N(\varepsilon) \in\N$ such that for all $(\io,t)\in F$ and $n\geq N$, 
\begin{equation}\label{eq-marstrand}
\frac{1}{n} H_n(\pi_0 (\pi_{\mathtt{x}}\mu\times S_t \mu_\io)) \geq \min \lbrace 1,\dimh\mu\rbrace - \varepsilon.
\end{equation}
 Thus, by Claim \ref{claim-equidistribution}, Birkhoff's ergodic theorem and Egorov's theorem, for $\delta>0$ small enough with respect to $\varepsilon$ there exists $m(\delta)\in\N$ and a set $V\subseteq \Gamma^\N\times Z$ such that $\bmu\times(\bmu\times\mathcal{L})_Z(V)\geq 1-\delta$ and 
\begin{equation}\label{eq-entropybound}
\frac{1}{m}\sum_{\ell=1}^m \frac{1}{n} H_n(G(\ko, t-\log|\lambda_1(\jo|_\ell)|)) \geq \min\lbrace 1,\dimh\mu\rbrace-\varepsilon
\end{equation}
for every $(\jo, (\ko, t))\in V$ and $m\geq m(\delta)$. The remainder of the proof is dedicated to showing that \eqref{eq-entropybound} in fact holds for every $t\in[0,M]$, where $M>0$ is as in the statement of the proposition. 

Let $(\mathcal{T}_s)_{s\geq 0}$ also denote the action of non-negative reals on $\Gamma^\N \times Z$ given by $\mathcal{T}_s(\jo, (\ko, t)) = (\jo, (\ko,t+s))$. Then $\bmu \times (\bmu\times\mathcal{L})_Z$ is invariant under $(\mathcal{T}_s)_{s\geq0}$, in particular, $\bmu\times(\bmu\times\mathcal{L})_Z(\bigcap_{\ell=1}^{\lfloor M/\sqrt{\delta}\rfloor} \mathcal{T}_{\ell\sqrt{\delta}}^{-1}V) \geq 1-M\sqrt{\delta} =: 1-\delta'$, so by Markov's inequality, there exists a set $V_{\delta}\subseteq \Gamma^\N\times \Gamma^\N$ such that $\bmu\times\bmu(V_{\delta})\geq 1-\sqrt{\delta'}$ and for every $(\jo,\ko)\in V_{\delta}$, 
$$
\mathcal{L}(\lbrace t\in [0,-\log|\lambda_2(\ko|_1)):\ (\jo,(\ko, t))\in \bigcap_{\ell=1}^{\lfloor M/\sqrt{\delta}\rfloor} \mathcal{T}_{\ell\sqrt{\delta}}^{-1}V\rbrace) \geq (1-\sqrt{\delta'}) \log|\lambda_2(\ko|_1)|^{-1}.
$$
Let now $\theta\in[0,M]$, and let $k$ be such that $k\sqrt{\delta} \leq \theta \leq (k+1)\sqrt{\delta}$. If $\delta$ is small enough with respect to $n$, then for every $(\jo,\ko)\in V_{\delta}$ there exists $0<t(\jo,\ko)<2^{-n}$ such that $(\jo,(\ko,t(\jo,\ko)))\in \mathcal{T}_{k\sqrt{\delta}}^{-1}V$, that is, \eqref{eq-entropybound} holds with $t=k\sqrt{\delta} + t(\jo,\ko) \in [\theta-2^{-n}, \theta+2^{-n}]$. By Lemma \ref{lemma-continuityofentropy},
$$
|H_n(G(\ko, k\sqrt{\delta} + t(\jo,\ko))) - H_n(G(\ko, \theta))| \leq 1,
$$
so \eqref{eq-entropybound} holds also for $t = \theta$ by adding another $\varepsilon$ to the right-hand side. In particular, applying \eqref{eq-Gproperty} to \eqref{eq-entropybound} with $t=\theta$ and $(\jo,\ko) \in V_\delta$, we obtain the statement with $U := V_\delta$.
\end{proof}

\begin{proof}[Proof of Proposition \ref{prop-productprojections}]
Let $\varepsilon,M>0$, let $n\in\N$ be large, let $\delta>0$ be small and let $m(\delta)\in\N$ be the integer and $U\subseteq\Gamma^\N\times\Gamma^\N$ the set given by Proposition \ref{prop-localvariant} with $\bmu\times\bmu(U)\geq 1-\delta$. Disintegrating $\bmu\times\bmu = \int \bmu\times\bmu_\io\,d\bmu(\io)$ and applying Markov's inequality, we find a set $X\subseteq\Gamma^\N$ with $\bmu(X)\geq 1-\sqrt{\delta}$ such that for every $\io\in X$, $\bmu\times\bmu_\io(U)\geq 1-\sqrt{\delta}$. Letting eventually $\delta\to 0$ along a countable sequence, it suffices to prove the statement of the proposition for $\io\in X$. 

Let $\theta \in [0,M]$, $\io\in X$ and $m\geq m(\delta)$. Similarly as in the proof of \cite[Lemma 3.4]{Hochman2014}, we will bound the entropy of $\pi_\theta(\pi_{\mathtt{x}}\mu\times\mu_\io)$ from below by the average of the entropies of the components $\pi_\theta \tilde{\Pi}(\bmu \times \bmu_\io)_{F_\ell^\theta(\jo)}$, $\jo\in\Gamma^\N\times\Gamma_\io$. Using the estimate 
$$
|H(\nu, \mathcal{D}_{m+p}|\mathcal{D}_p) - H_m(\nu)| \leq O(p)
$$
which holds for any integer $p$ and probability measure $\nu$ by Lemma \ref{lemma-continuityofentropy}, and the chain rule, we have 
    \begin{align}\label{eq-chainruleapplication}
        n H_{nm}(\pi_\theta(\pi_{\mathtt{x}}\mu\times\mu_\io)) &\geq \sum_{p=0}^n H(\pi_\theta (\pi_{\mathtt{x}}\mu\times\mu_\io), \mathcal{D}_{nm+p}|\mathcal{D}_p) - O(n^2) \nonumber \\
        &= \sum_{\ell=0}^{nm} H(\pi_\theta(\pi_{\mathtt{x}}\mu\times\mu_\io), \mathcal{D}_{\ell+n}|\mathcal{D}_{\ell}) - O(n^2)
    \end{align}
for every $\ell\in\N$. For $\jo\in\Gamma^\N$ and $\ell\in\N$, define the stopping time 
$$
\kappa(\jo,\ell) = \min \lbrace n:\ \lambda_1(\jo|_n) \leq 2^{-\ell}\rbrace
$$
so that $\pi_\theta \tilde{\Pi}(F_{\kappa(\jo,\ell)}(\jo,\ko)) \subset \R$ is an interval of length at most $2^{-\ell+1}$. Disintegrate $\bmu \times \bmu_\io = \int (\bmu \times \bmu_\io)_{F_{\kappa(\jo,\ell)}(\jo,\ko)}\,d\bmu \times \bmu_\io(\jo,\ko)$ for each $\ell$ and recall that $\pi_\mathtt{x}\mu \times \mu_\io = \tilde{\Pi}(\bmu \times \bmu_\io)$. Using \eqref{eq-chainruleapplication} and concavity of entropy, we have
    \begin{align}\label{eq-chainrule+concavity}
        &H_{nm}(\pi_\theta(\pi_{\mathtt{x}}\mu\times\mu_\io)) \nonumber\\
        \geq\ &\sum_{\ell=1}^{nm} \frac{1}{n}H(\pi_\theta(\pi_{\mathtt{x}}\mu\times\mu_\io), \mathcal{D}_{\ell+n}|\mathcal{D}_{\ell}) - O(n) \nonumber\\
        =\ &\sum_{\ell=1}^{nm}\frac{1}{n} H\left(\int\pi_\theta \tilde{\Pi}(\bmu \times \bmu_\io)_{F_{\kappa(\jo,\ell)}(\jo,\ko)}\,d\bmu \times \bmu_\io(\jo,\ko), \mathcal{D}_{\ell+n}|\mathcal{D}_{\ell}\right) - O(n)\nonumber\\
        \geq\ &\int \sum_{\ell=1}^{nm} \frac{1}{n}H(\pi_\theta\tilde{\Pi}(\bmu \times \bmu_\io)_{F_{\kappa(\jo,\ell)}(\jo,\ko)}, \mathcal{D}_{\ell+n}|\mathcal{D}_{\ell}) \,d\bmu \times \bmu_\io(\jo,\ko)- O(n).
    \end{align}
Here the support of $\pi_\theta \tilde{\Pi} (\bmu \times \bmu_\io)_{F_{\kappa(\jo,\ell)}(\jo,\ko)}$ may intersect at most $2$ intervals of $\mathcal{D}_\ell$ for each $\ell=1,\ldots, nm$. On the other hand, $\lambda_1(\jo|_{{\kappa(\jo,\ell)}})$ is comparable to $2^{-\ell}$ by the definition of $\kappa(\jo,\ell)$, so by introducing an error $O(1/n)$ to each term in the sum, \eqref{eq-chainrule+concavity} is bounded from below by
    \begin{align}\label{eq-lowerbound}
        \int \sum_{\ell=1}^{nm} \frac{1}{n}H_n( \pi_\theta S_{-\log|\lambda_1(\jo|_{{\kappa(\jo,\ell)}})|} \tilde{\Pi} (\bmu \times \bmu_\io)_{F_{\kappa(\jo,\ell)}(\jo,\ko)})\,d\bmu \times \bmu_\io(\jo,\ko)& \\  & \hspace{-4cm} - O(n+m) \nonumber.
    \end{align}
For $n$ and $m$ are large enough so that $\frac{n+m}{nm}\leq \varepsilon$, inserting the lower bound \eqref{eq-lowerbound} into \eqref{eq-chainrule+concavity} yields that 
\begin{align}\label{eq-lowerbound2}
    &\frac{1}{nm} H_{nm} (\pi_\theta(\pi_{\mathtt{x}}\mu\times\mu_\io)) \nonumber\\
    \geq\ &\int \sum_{\ell=1}^{nm} \frac{1}{n}H_n( \pi_\theta S_{-\log|\lambda_1(\jo|_{{\kappa(\jo,\ell)}})|} \tilde{\Pi} (\bmu \times \bmu_\io)_{F_{\kappa(\jo,\ell)}(\jo,\ko)})\,d\bmu \times \bmu_\io(\jo,\ko) - O(\varepsilon).
\end{align}
Since 
$$
\min_{i\in\Gamma}\log(|\lambda_1(i)|^{-1}) \leq \frac{\ell}{\kappa(\jo,\ell)}\leq \max_{i\in\Gamma}\log(|\lambda_1(i)|^{-1})
$$
for every $\jo\in\Gamma^\N$ and $\ell\in\N$, Proposition \ref{prop-localvariant} states that
$$
\frac{1}{nm}\sum_{\ell=1}^{nm} \frac{1}{n}H_n( \pi_\theta S_{-\log|\lambda_1(\jo|_{{\kappa(\jo,\ell)}})|} \tilde{\Pi} (\bmu \times \bmu_\io)_{F_{\kappa(\jo,\ell)}(\jo,\ko)}) \geq \min\lbrace 1,\dimh\mu\rbrace - O(\varepsilon)
$$
for every $(\jo,\ko)\in U\cap (\Gamma^\N\times\Gamma_\io)$ (recall that $U$ is the set given by Proposition \ref{prop-localvariant}). Here  the implicit constant in $O(\varepsilon)$ depends only on the ratio of $\kappa(\jo,\ell)$ and $\ell$ which in turn depends only on the numbers $\lambda_1(i)$ and $\lambda_2(i)$. Plugging this into \eqref{eq-lowerbound2} and recalling that $\bmu\times\bmu_\io(U) \geq 1-\sqrt{\delta}$ since $\io\in X$, we obtain that
$$
\frac{1}{nm}H_{nm}(\pi_\theta(\pi_{\mathtt{x}}\mu\times\mu_\io))\geq \min\lbrace 1,\dimh\mu\rbrace -O(\varepsilon)
$$
if $\delta$ is small enough with respect to $\varepsilon$. Setting $N_0= nm$, this is what we wanted to prove.
\end{proof}

\bibliographystyle{abbrv}
\bibliography{Bibliography}

\end{document}